\documentclass[12pt,a4paper]{amsart}
\usepackage{amsmath,amssymb,amsthm,amscd}

\usepackage[margin=1in,footskip=0.25in]{geometry}

\usepackage[breaklinks=true,colorlinks=true,linkcolor=blue,citecolor=red,urlcolor=blue]{hyperref}
\usepackage{graphicx,psfrag}
\usepackage{footnote}
\usepackage{mathrsfs} 

\usepackage{mathabx}
\usepackage{stmaryrd}
\usepackage{enumerate}
\usepackage{color}
\usepackage[all,cmtip]{xy}
\usepackage{hyperref}
\usepackage{mathrsfs} 

\usepackage{color}
\def\red{\color{red}}

\newtheorem{Thm}{Theorem}[section]
\newtheorem{Lem}[Thm]{Lemma}
\newtheorem{Pro}[Thm]{Proposition}

\newtheorem*{Con*}{Conjecture}

\theoremstyle{definition}

\theoremstyle{remark}
\newtheorem{Rem}[Thm]{Remark}

\newcommand{\Li}{\mathscr{L}} 
\newcommand{\sca}{\mathrm{scal}} 
\newcommand{\D}{\mathrm{D}}

\newcommand{\rd}{\mathrm{d}}

\newcommand{\la}{\lambda}

\newcommand{\J}{\mathrm{J}} 
\renewcommand{\phi}{\varphi}

\newcommand{\id}{\operatorname{id}}

\newcommand{\Ric}{\operatorname{Ric}}
\newcommand{\tr}{\operatorname{tr}}
\newcommand{\Id}{\operatorname{Id}}

\newcommand{\dd}{\mathrm{D}}

\def \ric {\mathrm{ric}} 

\begin{document}


\title
[Higher order obstructions to Riccati-type equations]
{Higher order obstructions to Riccati-type equations }

\author{Jihun Kim, Paul-Andi Nagy \and JeongHyeong Park}
\address{Jihun Kim and JeongHyeong Park\\
Department of Mathematics, Sungkyunkwan University, 16419 Suwon, South Korea}
\email{jihunkim@skku.edu, parkj@skku.edu}
\address{Paul-Andi Nagy \\
Center for Complex Geometry, Institute for Basic Science (IBS)\\
55 Expo-ro, Yuseong-gu, 34126 Daejeon, South Korea
}
\email{paulandin@ibs.re.kr}

 \begin{abstract}
We develop new techniques in order to deal with Riccati-type equations, subject to a further algebraic constraint, on Riemannian manifolds $(M^3,g)$. We find that the obstruction to solve the aforementioned equation has order $4$ in the metric coefficients 
and is fully described by a homogeneous polynomial in $\mathrm{Sym}^{16}TM$. Techniques from real algebraic geometry, reminiscent of 
those used for the ``PositiveStellen-Satz" problem, allow determining the geometry in terms of the connection coefficients.
Analysis of the latter 
shows flatness for the metric $g$; in particular we complete the classification
of asymptotically harmonic manifolds of dimension $3$, establishing those 
are either flat or real hyperbolic spaces.

\end{abstract}

\dedicatory{Dedicated to the memory of Professor Joseph A. Wolf who sadly passed away on August 14, 2023.}

\maketitle
\tableofcontents
\section{Introduction}
Let $(M^n,g)$ be Riemannian and let $SM$ denote the associated sphere bundle. Consider solutions $u=u(t) \in \mathrm{End}(T_{\gamma(t)}M)$ to the Riccati equation 
\begin{equation} \label{main11i}
u^{\prime}(t)+u^2(t)+\J(t)=0
\end{equation}
where $\gamma=\gamma(t)$ is the (possibly short time) geodesic through a point $(p,v)$ in the sphere bundle $SM$, that is $\gamma(0)=p, \gamma^{\prime}(0)=v$, and $\J(t)=\J(\gamma^{\prime}(t))$ is the Jacobi operator along $\gamma$; see the body of the paper for detailed definitions and conventions. The reader 
is referred to e.g. \cite{CGH, Wan, Wil,Fer} for further geometric set-ups where Riccati-type equations play a prominent role.

This paper wishes to address the question under which conditions 
on the metric $g$ a short time solution to \eqref{main11i} exists through any point $(p,v)$ in $SM$, provided the additional constraints 
\begin{equation} \label{alg-m}
u(t) \in \mathrm{Sym}^2_0\left ( T_{\gamma(t)}M \right ) \ \mathrm{and} \ u(t)\gamma^{\prime}(t)=0
\end{equation}
are imposed. Here $\mathrm{Sym}^2_0(TM)$ indicates the bundle of tensors which are symmetric with respect to the metric $g$ and which are also trace-free. When $n=2$ the requirements in \eqref{alg-m} clearly force $u=0$; hence the first case of interest is when 
$n \geq 3$ which will be treated in full in this work.

The main result of this work reads as follows.
\begin{Thm} \label{main-intro0}
Let $(M^3,g)$ be Riemannian and such that for any $(p,v) \in SM$ there exists a (possibly short time) solution $u=u(t)$ to the Riccati equation \eqref{main11i} which further satisfies the algebraic constraint \eqref{alg-m}. Then the metric $g$ is flat.
\end{Thm}
The proof uses a blend of real algebraic and differential geometry, articulated around the following key steps. 
\begin{itemize}
\item[$\bullet$] we eliminate $u$ from the \eqref{main11i} by taking into account \eqref{alg-m}; indeed 
the latter ensures that the only algebraic invariant for $u$ is $\tr(u^2)$. Thus differentiating in \eqref{main11i} we obtain 
an intrinsic obstruction involving only the covariant derivatives $\nabla^k\ric$ for $0 \leq k \leq 2$. This obstruction lives in $\mathrm{Sym}^{16}TM$ and reads 
\begin{equation} \label{obs-intro1}
\mathrm{P}^2=\D\left (-\mathrm{D}_1^2-4\tr(\ring{\J} \circ \ring{\J})\ric(X,X) \right)
\end{equation}
for all $X \in TM$. Here the polynomial invariants are given by  
\begin{equation*}
\begin{split}
\rm D_2:=&2\tr (\J(X) \circ \J(X))-(\ric(X,X))^2+(\nabla^2_{X,X}\ric)(X,X)\\
=&2\tr (\ring{\J}(X) \circ \ring{\J}(X))+(\nabla^2_{X,X}\ric)(X,X)\\
\rm D_1:=&(\nabla_X \ric)(X,X),\ \ \D:=\det (\ring{\J}(X) \circ \ring{\J}^{\prime}(X)-\ring{\J}^{\prime}(X) \circ \ring{\J}(X))
\end{split}
\end{equation*}
and $\mathrm{P}:=\tr(\ring{\J} \circ \ring{\J})\dd_2-\tr(\ring{\J} \circ 
\ring{\J}^{\prime})\dd_1.$ Note that the degrees of the polynomials involved satisfy $\deg \dd_2 \leq 4, \deg \dd_1 \leq 3,\deg \mathrm{D} \leq 10, \deg \mathrm{P} \leq 8$ which explains why \eqref{obs-intro1} has order at most $16$. Also see the body of the paper for more detailed information.
\item[$\bullet$] the observation that the Ricci tensor $\Ric$ is degenerate, in the sense of having $0$ as an eigenvalue 
allows considering $3$ types of distinguished directions in $TM$. We fully provide the list of algebraic 
solutions to \eqref{obs-intro1} for these directions.
\item[$\bullet$] finally, by interpreting the algebraic solutions found above in terms of the local geometry we obtain flatness for the metric $g$.
\end{itemize}
Note the result above is purely {\it{local}} in the sense it does not require any assumptions on the metric $g$, such as completeness.

We believe that with considerably more work these techniques can be extended to cover the four dimensional set-up; however both 
the algebraic and differential degrees of the obstruction should be expected to be much higher in that situation.
\subsection{Further motivation and background} \label{intro11}
 Let $(M,g)$ be a complete, simply connected Riemannian
manifold without conjugate points.  We denote the unit tangent
bundle of $M$ by $SM$. For $(p,v) \in SM$, let $\gamma_{v}$ be the geodesic through $p$ such that 
${\gamma'_{v}}(0) = v$ and let $b_{v} (x) =\displaystyle
\lim_{t\to \infty} (d (x,\gamma_v(t)) - t)$ be the corresponding
\emph{Busemann function} for $\gamma_{v}$. The level sets
of the Busemann function are called \emph{horospheres} of $M$.

 A complete, simply connected Riemannian manifold without conjugate points is called \emph{asymptotically harmonic} if the mean curvature of its horospheres is a universal constant, that is, if
its Busemann functions satisfy $\Delta b_v \equiv h$ for all $ v
\in SM$, where $h$ is a non-negative constant. Then $b_v$ is a
smooth function on $M$ for all
 $v$ and all the  horospheres of $M$ are smooth, simply connected
hypersurfaces in $M$ with constant mean curvature $h$.

On the other hand, a Riemannian manifold is called (locally) \emph{harmonic}, if about any point all the geodesic spheres of sufficiently small radii are of constant mean curvature. Since a harmonic manifold is Einstein, harmonic manifolds of dimensions $2$ and $3$ are of constant sectional curvature.  In 1944, Lichnerowicz \cite{Li} showed that a $4$-dimensional harmonic manifold is locally symmetric and conjectured that a harmonic manifold is flat or {locally} rank-one symmetric space (this conjecture is called \emph{Lichnerowicz's conjecture}). Nikolayevsky \cite{Ni} proved Lichnerowicz's conjecture in dimension $5$. Szab\'o \cite{Sz} proved the conjecture for compact harmonic manifolds.
On the other hand, Damek and Ricci \cite{DR} constructed nonsymmetric harmonic manifolds of dimension $\ge 7$, which are called \emph{Damek-Ricci spaces}.

It follows from \cite{RS1} that every complete, simply connected harmonic
manifold without conjugate points is asymptotically harmonic. It is natural to ask whether asymptotically harmonic manifolds are locally {symmetric.} Heber \cite{He} proved that for noncompact, simply connected homogeneous space, the manifold is asymptotically harmonic and Einstein if and only if it is flat, or a rank-one symmetric space of noncompact type, or a nonsymmetric Damek-Ricci space. For more characterizations of asymptotically harmonic manifolds, we refer to \cite{KP, Z}. By the Riccati equation \eqref{main11i}, one can easily check that a Ricci-flat asymptotically harmonic manifold is also flat.

In \cite{RS}, it was shown that harmonic manifolds with minimal horospheres are flat. As far as we know the classification 
of asymptotically harmonic manifolds with minimal horospheres is an open problem in dimension $3$(see also Remark \ref{err2}).
\begin{Rem} \label{err1}
In \cite{S2},  it was claimed that an asymptotically harmonic manifold of dimension $3$ with minimal horospheres is flat.
In the proof of \cite[Lemma 2.2]{S2} it is claimed that $\tr\sqrt{-R(x,v)v}=0$ implies $R(x,v)v=0$. However this is erroneous 
as the operator $-\J(v):v^{\perp} \to v^{\perp}$ does not have a sign in general hence it does not admit a square root. The new techniques developed in this paper
fully address this issue and also show that in order to fully understand the Riccati equation \eqref{main11i} one needs to examine 
higher order invariants in the Ricci tensor, not merely zero order ones.
\end{Rem}
\vspace{0.15in}

Below we recall some of the main background facts on asymptotically harmonic manifolds (see \cite{H, S, S2}). For $v\in SM$ and $x\in v^\perp$, we define $u^{\pm}(v)\in {\rm End}(v^\perp)$ by
\begin{equation*}
    u^+(v)(x)=\nabla_x \nabla b_{-v},\quad u^{-}(v)(x)=-\nabla_x\nabla b_{v}.
\end{equation*}
Record that $u^{\pm}(v)\in {\rm End}(v^\perp)$ follows from having $\Vert \mathrm{grad}\ b_{\pm v} \Vert=1$.

Then $u^{\pm}(t):=u^{\pm}(\varphi^t v)$ satisfy the Riccati equation \eqref{main11i}, where $\varphi^t: SM\to SM$ is the geodesic flow of $g$. Here, $u^{+}(t)$ and $u^{-}(t)$ are called the \emph{unstable} and \emph{stable} Riccati solutions, respectively.
Using that $\tr u^+ (v) =\Delta b_{-v} = h$ and $\tr u^{-}(v) = -\Delta b_v = -h$ for all $v\in SM$ we see that $u^{\pm}$ also satisfy the algebraic constraints in \eqref{alg-m}, provided that $h=0$.

From the Riccati equation~\eqref{main11i} we clearly see that any $2$-dimensional asymptotically harmonic manifold
is either a flat space or a real hyperbolic plane of constant curvature $-h^2$.
This shows that the study of asymptotically harmonic manifold  {begins} with dimension $3$.
 Towards this, it was shown in  \cite{H} that any Hadamard
 asymptotically harmonic manifold of bounded sectional curvature, satisfying  {some} mild hypothesis
 on  {the}  curvature tensor is a real hyperbolic space of constant sectional curvature $-\frac{h^2}{4}$.
 Finally, this result was improved by Schroeder and Shah \cite{S} by relaxing the hypothesis on the curvature tensor.

The complete classification of asymptotically harmonic manifolds of dimension $3$ follows now directly from Theorem \ref{main-intro0}.
\begin{Thm}\label{class}
Let $(M^3,g)$ be asymptotically harmonic. Then $(M^3,g)$ is flat if $h=0$ or a real hyperbolic space of constant sectional curvature $-\frac{h^2}{4}$ if $h>0$.
\end{Thm}
\begin{proof}
If $(M^3,g)$ has minimal horospheres, that is $h=0$, both \eqref{main11i} and \eqref{alg-m} are satisfied for $u^{\pm}$. That 
$g$ is flat follows then from Theorem \ref{main-intro0}. When $h>0$ then $(M^3,g)$ is a real hyperbolic space of constant sectional curvature $-\frac{h^2}{4}$ by \cite{S}.
\end{proof}
To finish, we propose the following conjecture which extends Lichnerowicz's conjecture
on harmonic manifolds to the realm of asymptotically harmonic manifolds.

\begin{Con*}
    Let $(M,g)$ be a complete, simply connected Riemannian
manifold without conjugate points. If $(M,g)$ is asymptotically harmonic, then $M$ is either flat or rank-one symmetric space of noncompact type.
\end{Con*}
This conjecture has been partially resolved to date. Including Theorem~\ref{class}, we refer the reader to \cite{CS, He, H, KP, S, Z}.
\subsection{Outline of the paper} \label{intro13}Section \ref{prel-conv} contains some preliminaries from Riemannian geometry and also a detailed proof 
of having $\Ric$ degenerate under the assumptions in Theorem \ref{main-intro0}. The section ends with the crucial 
observation that one can assume $\Ric$ have rank $1$ on some open dense subset of the $3$-dimensional manifold under scrutiny. In Section \ref{12} we first give a proof of the intrinsic obstruction in \eqref{obs-intro1}. To understand its algebraic content 
we diagonalise, locally, the Ricci tensor with eigenfunctions $0,\lambda_2, \lambda_3$ and local eigenframe $e_1,e_2,e_3$. For directions of type $xe_1+ye_2, xe_1+ye_3$ respectively $xe_2+ye_3$, where $x,y \in \mathbb{R}^2$ we compute explicitely 
the polynomials  $\D,\D_1,\D_2$ respectively $\mathrm{P}$. Considering these special directions has the advantage of reducing the 
homogeneous constraint in \eqref{obs-intro1} to a polynomial one in $\mathbb{R}[t]$.

The resulting quadratic polynomial equations are contained in 
Proposition \ref{pol-1} respectively Proposition \ref{pol-1bis} and read  
\begin{equation} \label{P-eqn-i}
\mathrm{P}^2+(t^2+1)({\rm d}_1\mathbf{c})^2=(t^2+1)(\mathbf{a}\mathbf{c})^2\mathbf{q}
\end{equation} 
for real valued polynomials of degrees $\deg \mathrm{P} \leq 5, \deg \rm d_1 \leq 2, \deg \mathbf{c} \leq 2, \deg \mathbf{a}=2$ and 
where  $\rm \mathrm{P}=\mathbf{a}\rm d_2-\mathbf{a}_1 \rm d_1$. In addition the polynomial $\mathbf{q}$ satisfies $\deg \mathbf{q}=0$ or 
$\deg \mathbf{q}=2$ and is entirely explicit in terms of the eigenfunctions $\lambda_2, \lambda_3$. The full list of solutions to \eqref{P-eqn-i} is given in subsection \ref{pol-s}. 

In section \ref{diff} we fully explicit the polynomials in \eqref{main-intro0}, this time in terms of the covariant derivatives 
of the Ricci tensor; we show that the algebraic form of the solutions to \eqref{P-eqn-i} fully determines the connection coefficients 
of the Levi-Civita connection on the open set in $M$ where $\Ric$ has rank $2$ and $\lambda_2 \neq \lambda_3$. After some algebraic computations, we prove that the metric $g$ must be flat therein which shows that such an open subset must be empty. The remaining case is thus  
when $\lambda_2=\lambda_3$, or equivalently when the Riemann nullity of $g$ has rank one; see \cite{Bl,MoSe} for some general 
results on metrics with non-trivial Riemann nullity. This type of geometry is dealt with in a similar algebraic way in section \ref{S5} where we show it corresponds to a class of Hessian-type equations on a Riemann surface which turn out to admit no solution.
\subsection*{Acknowledgements}
The research of Paul-Andi Nagy was supported by the Institute for Basic Science (IBS-R032-D1).
The research of J. H. Park work was supported by the National Research Foundation of Korea (NRF) grant funded by the Korea government (MSIT) (RS-2024-00334956). J. H. Park would
like to thank H. Shah for drawing her attention to the classification problem of asymptotically harmonic manifolds and also the Department of Mathematics, Stony Brook University, for hospitality during the initial stages of this project. It is a pleasure to thank G. Knieper and N. Peyerimhoff for valuable exchanges on an earlier version of this work.
\section{Preliminaries and rank structure } \label{prel-conv}
\subsection{Conventions from Riemannian geometry} \label{conv-R}
Let $(M^n,g)$ be Riemannian; we will frequently use the metric to identify vectors and one forms via the induced bundle isomorphism 
$g:TM \to \Lambda^1M$.
 Whenever $X,Y$ belong to $TM$ we indicate with $X \wedge Y $ the endomorphism which maps Z to $g(Y,Z)X - g(X,Z)Y$; furthermore, for an orthonormal basis $\{e_i\}$, we denote $e^{ij}:=e_i \wedge e_j$. Letting $\nabla$ be the Levi-Civita connection 
of $g$ the Riemann curvature tensor is defined according to $R(X,Y):=\nabla_{X,Y}^2-\nabla^2_{Y,X}:TM \to TM$; the Ricci form 
$\ric : TM \times TM \to \mathbb{R}$ is given by $\ric(X,Y)=-\sum_i R(e_i,X,e_i,Y)$ and the Ricci operator $\Ric:TM \to TM$ is determined from 
$\Ric=g^{-1}\ric$. 
Below we use the shorthand notation  $\J(v)=R(\cdot,v)v$ with  $v \in TM$ for the Jacobi operator and $\J(t)=\J(\gamma^{\prime}(t))$ where $\gamma$ is a curve on $M$ tangent to $v$.
 
Now recall that in dimension $3$ we have $R=\rho \wedge g$ where $\rho=\Ric-\frac{\rm scal}{4}\id$  is the Schouten tensor, $\sca:=\tr \Ric$ is the scalar curvature, and the action in the Kulkarni-Nomizu product above takes into account the identification of $1$-forms and vectors via the metric. Explicitly 
\begin{equation} \label{R-p}
R(X,Y)=\Ric X \wedge Y+X \wedge \Ric Y-\frac{\rm scal}{2} X \wedge Y
\end{equation}
whenever $X,Y \in TM$.
 
These sign conventions are slightly different from the usual ones, see e.g. \cite{Besse}, but are used in this paper in order to easily relate 
to pre-existing work. For further use we also record a few more general facts as follows. 
\begin{Pro} Letting $(M^3,g)$ be Riemannian we have
    \begin{equation} \label{J2}
\tr(\J(v) \circ \J(v))=\left ( \tr(\rho^2)g(v,v)+2\tr(\rho)g(\rho v, v)-2g(\rho v,\rho v) \right) g(v,v) + g(\rho v,v)^2
\end{equation}
for all $v \in TM$. 
\end{Pro}
\begin{proof}
    A direct algebraic computation based on \eqref{R-p} yields
    \begin{align*}
        \tr(\J(v) \circ \J(v)) &= \sum_{i,j=1}^3 g(\J(v)e_i,e_j)^2\\
        &=\sum_{i,j=1}^{3}\Big\{g(v,v)g(\Ric e_i,e_j) - g(\Ric e_i, v)g(v,e_j) \\
        &\qquad\qquad+ g(\Ric v, v)g(e_i,e_j)-g(v,e_i)g(\Ric v,e_j)\\
        &\qquad\qquad -\frac{\sca}{2}\big(g(v,v)g(e_i,e_j)-g(v,e_i)g(v,e_j)\big)\Big\}^2\\
        &=\{\tr(\Ric^2)g(v,v)-2g(\Ric v, \Ric v) \\
        &\qquad-\frac{\sca^2}{2}g(v,v)+\sca \,g(\Ric v,v)\}g(v,v) + g(\Ric v ,v)^2.
    \end{align*}
Replacing $\Ric$ with $\rho+\frac{\sca}{4}\id$, we obtain~\eqref{J2}.
\end{proof}
Lastly, recall that the divergence operator $\delta : \Gamma (\mathrm{End}(TM)) \to \Gamma(TM)$ is defined according to 
$\delta S:=-\sum (\nabla_{e_i}S)e_i \in \Gamma(TM)$, where $\{e_i\}$ is some local orthonormal frame in $TM$. When $n=3$, case which we mainly deal with in this paper, the differential Bianchi identity 
reduces to  
\begin{equation} \label{B1}
        \delta \Ric = -\frac{1}{2}\mathrm{grad} (\sca).
\end{equation}
This will be systematically used in section \ref{12}. 
\subsection{Algebraic curvature structure} \label{alg-s}
We start with the following preliminary observations.
\begin{Lem}[\cite{KP, Z}]\label{cont}
If $(M, g)$ is an asymptotically harmonic
manifold, then the map $v \mapsto u^{\pm}(v)$ is continuous on $SM$.
\end{Lem}
Similarly to what is proved for asymptotically harmonic manifolds with minimal horospheres we have 
\begin{Lem}\label{vec}
Let $(M^n,g)$ be such that for any $(p,v) \in SM$ there exists a (possibly short time) solution $u=u(t)$ to  \eqref{main11i} satisfying \eqref{alg-m}. Then the Ricci form satisfies $\ric(v,v)\le 0$.
\end{Lem}
\begin{proof}
    From the Riccati equation~\eqref{main11i}, we get $\ric(v,v)=-\tr(u^2(v))$. Since $u(v)$ is symmetric, trace-free and acts 
    on a $(n-1)$-dimensional space, namely $v^{\perp}$, the claim follows.
   \end{proof}
In dimension $3$ substantially more information on the eigenvalue structure of the Ricci tensor is available. 
\begin{Lem}\label{lem:curvdiag}
Let $(M^3,g)$ be such that for any $(p,v) \in SM$ there exists a (possibly short time) solution $u=u(t)$ to  \eqref{main11i} satisfying \eqref{alg-m}. Then 
    $$\det(\Ric)=0 \ \mathrm{at \ any \ point \ of \ } M.$$ 
    \begin{proof}
We work at an arbitrary point $p \in M$. As $\dim M =3$, we can identify $S_p M$ with the $2$-sphere ${S}^2 \subseteq \mathbb{R}^3$, and $v^\perp$ with $T_v S^2$ for $v\in S_p M $. At the point $p$ the operator $u(v)$ is symmetric and trace-free; in other words $u(v)$ belongs to the space $\mathrm{Sym}^2_0(T_vS^2)$ of trace-free, symmetric tensors on $T_vS^2$. At the point $p \in M$, the correspondence $v \mapsto u(v)$ 
thus defines a section in the complex line bundle $\mathrm{Sym}^2_0(TS^2)$ over $S^2$. As it is well known, the latter bundle is not trivial, hence any of its sections must have a zero. It follows there exists $v_0 \in S_p M$ such that $u(v_0) =0$. 
   
  By the Riccati equation \eqref{main11i}, we obtain $u^{\prime}(v_0)(x) = -R(x,v_0)v_0$ for $x\in v_0^\perp$. Since  $\tr u(v) = 0$ for all $(p,v) \in SM$, it follows that  $\tr u^{\prime}(v_0)=0$. Consequently, tracing the Riccati equation shows that $\ric(v_0, v_0) = 0$. Since $\ric $ is non-positive the function $f:\mathbb{R} \to \mathbb{R}$ given by $f(t)=\ric(v_0+tw, v_0+tw)$ has a maximum at $t=0$ whenever $w \in T_pM$; that is $f(t) \leq f(0)=0$. Then $f^{\prime}(0)=2\ric(v_0,w)=0$ showing that $v_0 \in \ker \Ric$ and the claim is proved.
  \end{proof}
\end{Lem}
 Below we spell out the structure of Riemann curvature tensor on open sets as follows.

Since $\det(\Ric)=0$ by the previous Lemma, it follows that $\mathrm{rk} \Ric \leq 2$ at each point of $M$. By general topology 
we may find open sets $U_k, 0 \leq k \leq 2$ such 
\begin{equation*}
\mathrm{rk} \Ric=k \ \mathrm{on} \ U_k \ \mathrm{and} \ U_0 \cup U_1 \cup U_2 \ \mathrm{is \ dense \ in} \ M.
\end{equation*}
On these open sets we have the following 
\begin{itemize}
\item[$\bullet$] on $U_0$ the metric $g$ is flat 
\item[$\bullet$] over $U_1$ the distribution $\ker \Ric$ has rank $2$ hence $\Ric$ acts by multiplication 
with $\rm scal$ on $\ker^{\perp} \Ric$. By choosing, if necessary, a local trivialisation of the rank 
two bundle $\ker \Ric$ we find, around each point in $U_1$, local orthonormal frames $\{e_1,e_2,e_3\}$ such that 
$\Ric e_1=\Ric e_2=0 \ \mathrm{and} \ \Ric e_3=(\sca) e_3.$ Then 
\begin{equation} \label{K1}
R(e_1,e_2)=-\frac{{\rm scal}}{2}e^{12}, \ R(e_1,e_3)=\frac{{\rm scal}}{2}e^{13}, R(e_2,e_3)=\frac{\sca}{2}e^{23}
\end{equation}
according to \eqref{R-p}.
\item[$\bullet$] over $U_2$ the bundle $\mathrm{Im} \Ric$ has rank $2$, so after possibly choosing a local trivialisation therein we obtain local orthonormal frames $\{e_1,e_2,e_3\}$ which satisfy 
$\Ric e_1=0, \ \Ric e_2=\lambda_2 e_2, \ \Ric e_3=\lambda_3 e_3$. Again from \eqref{R-p} and $\sca=\lambda_2+\lambda_3$ we get 
\begin{equation} \label{K2}
R(e_1,e_2)=\frac{\lambda_2-\lambda_3}{2}e^{12}, \ R(e_1,e_3)=\frac{\lambda_3-\lambda_2}{2}e^{13} \ \mathrm{and} \ R(e_2,e_3)=\frac{\sca}{2}e^{23}.
\end{equation}
\end{itemize}
Note that the sectional curvature does not have a sign in either of the situations above.
\begin{Rem}\label{rmk1}
In the Lemma \ref{lem:curvdiag}, from the construction of $\lambda_2$ and $\lambda_3$ (using $\tr u^{\prime}(v)=0$ and $\ric(e_3,e_3)\le 0$), we note that the eigenvalues $\lambda_2$
 and $\lambda_3$ can be interchanged, but the aforementioned form of the curvature operator is the same.
 \end{Rem}

We conclude this section by analysing the Riccati equation to second order. 
\begin{Pro} \label{Rica-2}
Let $(M^3,g)$ be such that for any $(p,v) \in SM$ there exists a (possibly short time) solution $u=u(t)$ to  \eqref{main11i} satisfying \eqref{alg-m}. Then 
\begin{itemize}
\item[(i)] we have
\begin{equation} \label{1-jet1}
2\tr(u \circ \J)=\tr (\J^{\prime})=(\nabla_{\gamma^{\prime}(t)}\ric)(\gamma^{\prime}(t), \gamma^{\prime}(t))
\end{equation}
\item[(ii)] we have 
\begin{equation} \label{jet-2}
2\tr(u \circ \J^{\prime})=2\tr (\J \circ \J)-(\ric(\gamma^{\prime}(t),\gamma^{\prime}(t))^2+(\nabla^2_{\gamma^{\prime}(t),\gamma^{\prime}(t)}\ric)(\gamma^{\prime}(t),\gamma^{\prime}(t)).
\end{equation}
\end{itemize}
\end{Pro}
\begin{proof}
(i) Since $(u^2)^{\prime}=u^{\prime}\circ u+u \circ u^{\prime}$, 
taking the trace yields 
$$\tr((u^2)^{\prime})=2\tr(u^{\prime} \circ u)=-2\tr(u^3+u\circ \J)=-2\tr(u\circ \J).$$
To obtain the last equality we have used the Riccati equation and that 
$\tr(u^3)=0$ 
which is entailed by having 
$u$ trace-free and symmetric, acting on a $2$-dimensional space. 
As seen before we have 
$${\tr(u^2)}=-\tr(\J)
$$
and the claim follows after differentiation in \eqref{main11i}.\\
(ii) Differentiating in \eqref{1-jet1} we get 
\begin{equation*}
2\tr(u^{\prime}\circ \J)+2\tr(u \circ \J^{\prime})=\tr(\J^{\prime \prime}).
\end{equation*}
Plugging in the value for $u^{\prime}$ given by the Riccati equation whilst taking into account that 
$u^2=\frac{1}{2}\tr(u^2)\id=-\frac{\tr(\J)}{2}\id$ on $v^{\perp}$ leads to 
$2\tr(u\circ \J^{\prime})=2\tr (\J \circ \J)-(\tr \J)^2+\tr(\J^{\prime \prime})$. The claim follows from $\tr{\J}=\ric(v,v)$.
\end{proof}

\subsection{The rank $1$ case} \label{rk1}
In what follows we assume that $(M^3,g)$ is such that for any $(p,v) \in SM$ there exists a (possibly short time) solution $u=u(t)$ to  \eqref{main11i} satisfying \eqref{alg-m}. In addition we work on open subsets $U \subseteq U_1$ where the rank $2$ bundle $\ker \Ric$ is trivial. We also indicate with 
$\mathscr{H}$ the distribution orthogonal to the unit vector field $e_3$ and use the Riccati equation to show the following.  

\begin{Lem} \label{l1}
The following hold over $U_1$
\begin{itemize}
\item[(i)] $\mathscr{L}_{e_3}\sca=0$
\item[(ii)] $\mathrm{div} e_3=0$.
\end{itemize}
\end{Lem}
\begin{proof}
(i) A short algebraic computation based on \eqref{K1} shows that $\J(e_3)={\frac{\sca}{2}\Id}$ on $\mathscr{H}$ at each point in $U_1$. 
Now evaluate \eqref{1-jet1} at $t=0$, for the geodesic 
$\gamma(0)=m \in M, \gamma^{\prime}(0)=(e_3)_m$; since 
for these choices $\J$ is pure trace at $t=0$ it follows that $(\nabla_{e_3}\ric)(e_3,e_3)=0$ over $U$. Using 
the eigenvalue structure of $\Ric$ leads now directly to the claim.\\
(ii) Evaluate the differential Bianchi identity \eqref{B1} on $e_3$. Then using (i) we obtain
    \begin{align*}
        0 &= \frac{1}{2}\nabla_{e_3}\sca
       =-(\delta \Ric)(e_3)\\
        &=(\nabla_{e_1}\ric)(e_1, e_3)+ (\nabla_{e_2}\ric)(e_2, e_3) + (\nabla_{e_3}\ric)(e_3, e_3).
    \end{align*}
    As we know from the proof of (i), the last term of the right-hand side vanishes. At the same time we obtain
    \begin{align*}
        (\nabla_{e_1}\ric)(e_1, e_3)&=\nabla_{e_1}(\ric(e_1 , e_3 )) - \ric(\nabla_{e_1}e_1 , e_3 ) - \ric(e_1 , \nabla_{e_1} e_3 )
       =g(\nabla_{e_1}e_3,e_1)\sca.
    \end{align*} 
    Similarly, we get $(\nabla_{e_2}\ric)(e_2, e_3) =g(\nabla_{e_2}e_3,e_2)\sca$. Since $\sca$ is nowhere 
    vanishing in $U_1$ it follows that $e_3$ is divergence free, as claimed.
\end{proof}

The next set of obstructions comes from differentiating the Riccati equation to second order.
\begin{Pro} \label{U1}
The set $U_1=\emptyset$.
\end{Pro}
\begin{proof}
Evaluate the second derivative of the Riccati equation, that is \eqref{jet-2}, at $t=0$, for the geodesic 
${\gamma(0)=m\in M,\,} \gamma^{\prime}(0)=v_m$ where $v \in \mathscr{H}_m$ has length $1$, that is $g(v,v)=1$. As we have seen before $u_{+}(v)=0$ 
and $\tr \J(v)=\ric(v,v)=0$, so we are left with 
\begin{equation*}
-2\tr (\J(v) \circ \J(v))=(\nabla^2_{v,v}\ric)(v,v).
\end{equation*}
A short algebraic computation based on the eigenvalue structure of Ricci tensor and \eqref{J2} reveals that 
\begin{equation*}
\tr (\J(v) \circ \J(v))=\frac{\sca^2}{2}.
\end{equation*}
At the same time observe that $(\nabla_U\ric)(v,v)=0$ for all $U \in TM$ since $\Ric$ vanishes on $\mathscr{H}$. Thus, by extending, if necessary, $v$ to a local section of $\mathscr{H}$ we compute 
\begin{equation*}
\begin{split}
(\nabla^2_{v,v}\ric)(v,v)=&-(\nabla_{\nabla_vv}\ric)(v,v)-2(\nabla_v\ric)(\nabla_vv,v)=-2(\nabla_v\ric)(\nabla_vv,v)\\
=&2\ric(\nabla_vv,\nabla_vv)
=2\sca (g(\nabla_ve_3, v))^2.
\end{split}
\end{equation*}
In the computation above we have taken systematically into account that $\ric(\mathscr{H}, TM)=0$. As the scalar curvature 
function is nowhere vanishing in $U_1$ we conclude that 
\begin{equation*}
(g(\nabla_ve_3,v))^2=-\frac{\sca}{2}.
\end{equation*}
Now consider the tensor $g^{-1}\mathscr{L}_{e_3}g \in \mathrm{Sym}^2M$ and let $Q$ be its projection onto the sub-bundle $\mathrm{Sym}^2\mathscr{H} 
\subseteq \mathrm{Sym}^2M$. Note that $\tr(Q)=0$ since $\mathrm{div}(e_3)=0$ and 
$g(\nabla_{e_3}e_3,e_3)=0$. 

Then $(g(Qv,v))^2=-\frac{\sca}{2}$ for all $v \in \mathscr{H}$ with $g(v,v)=1$. We now show that 
forces $\sca=0$ and $Q=0$. Because $Q:\mathscr{H} \to \mathscr{H}$ is symmetric and trace-free it can be diagonalised as 
$Qw_1=\lambda w_1$ and $Qw_2=-\lambda w_2$ where $\{w_1,w_2\}$ form an orthonormal basis and $\lambda \in \mathbb{R}$. Then  if 
$v=xw_1+yw_2$ where $x^2+y^2=1$ we get 
$$ \lambda^2(x^2-y^2)^2=-\frac{\sca}{2}
$$
which is clearly impossible unless $\lambda=\sca=0$. Since assuming $U_1$ not empty guarantees that $\sca$ is nowhere 
vanishing in $U_1$, we have obtained a contradiction and the claim is proved.
\end{proof}
\section{The second order obstruction in $\rm Sym^{16}TM$} \label{12}
Motivated by the proof of Proposition \ref{Rica-2} we shall develop in this section a general tensorial obstruction 
to the existence of metrics, in dimension $3$, such that \eqref{main11i} and \eqref{alg-m} admit solutions through any point in $SM$.
This obstruction is second order in the derivatives of the Ricci tensor and turns out to be given by a homogeneous polynomial 
of order $16$, hence it lives in $\rm Sym^{16}TM$. The key idea is to eliminate $u$ from the constraints in Proposition \ref{Rica-2}, to which aim 
we proceed as follows. We indicate with $\ring{\J}^{\prime}(X)$ the trace-free component in $\J^{\prime}(X)$, that is 
$\ring{\J}^{\prime}(X):={\J}^{\prime}(X)-\frac{1}{2}\tr({\J}^{\prime}(X)) \id_{X^{\perp}}$.

Whenever $X$ in $TM$ we consider the polynomial invariants(or symmetric tensors) given by  
\begin{equation*}
\begin{split}
\rm D_2:=&2\tr (\J(X) \circ \J(X))-(\ric(X,X))^2+(\nabla^2_{X,X}\ric)(X,X)\\
=&2\tr (\ring{\J}(X) \circ \ring{\J}(X))+(\nabla^2_{X,X}\ric)(X,X)\\
\rm D_1:=&(\nabla_X \ric)(X,X)\\
\D:=&\det (\ring{\J}(X) \circ \ring{\J}^{\prime}(X)-\ring{\J}^{\prime}(X) \circ \ring{\J}(X)).
\end{split}
\end{equation*}
The degrees of these homogeneous polynomials are at most $4, 3$ and $10$, respectively.
\begin{Thm} \label{obs2}Let $(M^3,g)$ be such that for any $(p,v) \in SM$ there exists a (possibly short time) solution $u=u(t)$ to  \eqref{main11i} satisfying \eqref{alg-m}. Then 
\begin{equation} \label{C2}
\mathrm{P}^2=\D\left (-\mathrm{D}_1^2-4\tr(\ring{\J} \circ \ring{\J})\ric(X,X) \right)
\end{equation}
for all $X \in TM$, where the homogeneous polynomial $P$ with $\deg \mathrm{P} \leq 8$ is given by 
$$\mathrm{P}:=\tr(\ring{\J} \circ \ring{\J})\dd_2-\tr(\ring{\J} \circ 
\ring{\J}^{\prime})\dd_1.$$
\end{Thm}
\begin{proof}
We work at some point $m \in M$.
First we determine $u$ from the constraints in \eqref{1-jet1} and \eqref{jet-2} which we rephrase as 
\begin{equation*}
2\tr(u \circ \ring{\J})=\dd_1 \ \mathrm{and} \ 2\tr(u \circ \ring{\J}^{\prime})=\dd_2
\end{equation*}
since $u$ is trace-free. Since we are dealing with homogeneous polynomials we may assume that $g(X,X)=1$; with respect to some 
orthonormal basis in $X^{\perp}$ let 
$$\left ( \begin{array}{lr} a & b \\ b & -a \end{array} \right ), 
\left ( \begin{array}{lr} A & B \\ B & -A \end{array} \right ) \ \mathrm{respectively} \ \left ( \begin{array}{lr} A_1 & B_1 \\ B_1 & -A_1 \end{array} \right )$$ be the matrices of the trace-free symmetric endomorphisms $u, \ring{\J}$ respectively $\ring{\J}^{\prime}$. Then 
$$ aA+bB=\frac{1}{4}\mathrm{D}_1 \ \mathrm{and} \ aA_1+bB_1=\frac{1}{4}\mathrm{D}_2
$$
hence $4da=B\dd_2-B_1\dd_1$ and $4db=-A\dd_2+A_1\dd_1$ where $d:=A_1B-AB_1$.  It follows that 
\begin{equation*}
\begin{split}
16d^2(a^2+b^2)=&(B\dd_2-B_1\dd_1)^2+(A\dd_2-A_1\dd_1)^2\\
=&(A^2+B^2)\dd_2^2-2\dd_1\dd_2(AA_1+BB_1)+(A_1^2+B_1^2)\dd_1^2.
\end{split}
\end{equation*}
The discriminant of this, regarded as an equation in $\dd_2$ reads 
\begin{equation*}
\begin{split}
&4\dd_1^2(AA_1+BB_1)^2-4(A^2+B^2)\left ( (A_1^2+B_1^2)\dd_1^2-16d^2(a^2+b^2) \right )\\
=&4d^2(-\dd_1^2+16(A^2+B^2)(a^2+b^2)).
\end{split}
\end{equation*}
It follows that 
\begin{equation} \label{pyth-1}
\begin{split}
((A^2+B^2)\dd_2-\dd_1(AA_1+BB_1))^2=&d^2(-\dd_1^2+16(A^2+B^2)(a^2+b^2))\\
=&d^2(-\dd_1^2-8(A^2+B^2)\tr(\J)).
\end{split}
\end{equation}
after also taking into account $\tr(u^2)=2(a^2+b^2)$ and $\tr(u^2)=-\tr(\J)$. The rest of the invariants are captured according to $\tr(\ring{\J} \circ \ring{\J})=2(A^2+B^2), \tr(\ring{\J}\circ \ring{\J}^{\prime})=2(AA_1+BB_1)$ and 
$\tr(\ring{\J}^{\prime} \circ \ring{\J}^{\prime})=2(A_1^2+B_1^2)$. Finally, 
$4d^2=\det (\ring{\J} \circ \ring{\J}^{\prime}-\ring{\J}^{\prime} \circ \ring{\J})$ and the claim follows.
\end{proof}
This result says that the second order derivatives of the Ricci tensor, which are essentially encoded in $D_2$, depend polynomially on its first and zero order derivatives. Since $D \neq 0$ on a dense open set in $T^{\times}M=\{X \in  TM: X \neq 0\}$ a necessary condition for solving 
\eqref{C2} is having 
$$-\mathrm{D}_1^2-4\tr(\ring{\J} \circ \ring{\J})\ric(X,X) \geq 0.$$
This is precisely of PositivStellenSatz-type as covered by Hilbert's 17-th problem, solved by E. Artin in 1927, see \cite{Art}.

In fact, our set-up contains much stronger algebraic information. Observe that we may choose a basis, possibly depending on $X$, w.r.t. which $\ring{\J}(X)$ is diagonal, that is 
$B=0$. Then \eqref{C2} becomes 
\begin{equation} \label{C3}
A^2(A\dd_2-\dd_1A_1)^2+(AB_1\dd_1)^2=-8(A^2B_1)^2\ric(X,X).
\end{equation}
In case we know that the polynomial $A$ is not identically zero, the degree in the constraint above can be lowered; this procedure will be explained in the next section, together with the fact that $B_1^2$ is in general only a rational function. Since the polynomial $-\ric(X,X)$ is equally a sum of squares, methods of real algebraic geometry can be used to solve this polynomial constraint.
\subsection{Geometry when $\Ric$ has rank $2$} \label{rk2}
In what follows we assume that $(M^3,g)$ is such that for any $(p,v) \in SM$ there exists a (possibly short time) solution $u=u(t)$ to  \eqref{main11i} satisfying \eqref{alg-m}.
Furthermore we work over the open set $U_2$ where the Ricci tensor has rank $2$; accordingly the eigenvalue structure 
of the Ricci tensor, respectively the Riemann curvature tensor read as in \eqref{K2}. 


In order to obtain 
constraints on the geometry we will compute all invariants pertaining to the Jacobi operator and its first order derivatives. 
Below we consider the following special directions in $T^{\times}M$, for $(x,y) \in \mathbb{R}^2 \backslash \{0\}$ together with 
a choice of orthonormal basis $\{w_1,w_2\}$ in $v^{\perp}$ 
\begin{itemize}
\item[($a_1$)] $X=xe_1+ye_2$ and $w_1=e_3,w_2=(x^2+y^2)^{-\frac{1}{2}}(ye_1-xe_2)$
\item[($a_2$)] $X=xe_1+ye_3$ and $w_1=e_2,w_2:=(x^2+y^2)^{-\frac{1}{2}}(ye_1-xe_3)$ 
\item[($a_3$)] $X=xe_2+ye_3$ and $w_1=e_1, w_2=(x^2+y^2)^{-\frac{1}{2}}(ye_2-xe_3)$.
\end{itemize}
Furthermore, to match the notation in the proof of Theorem \ref{obs2} we write 
$\left ( \begin{array}{lr} A & B \\ B & -A \end{array} \right )$ for the matrix of $\ring{\J}(X)$ computed with respect to the basis $\{w_1,w_2\}$.
\begin{Lem} \label{jac-f}
For the choices of directions $X \in T^{\times}M$ and orthonormal 
basis $\{w_1,w_2\}$ in $X^{\perp}$ above the trace-free Jacobi operator $\ring{\J}(X)$ is diagonal, that is $B=0$, and satisfies 
\begin{itemize}
\item[(i)] $2A=(\lambda_3-\lambda_2)x^2+\lambda_3 y^2$ in case $(a_1)$
\item[(ii)] $2A=(\lambda_2-\lambda_3)x^2+\lambda_2 y^2$ in case $(a_2)$
\item[(iii)] $-2A=\lambda_3x^2+\lambda_2y^2$ in case $(a_3)$.
\end{itemize}
\end{Lem}
The proof is a straightforward algebraic computation based on \eqref{K2}. Observe that $A$ is, in all cases, non-identically 
zero at each point in $U_2$ since $\lambda_2<0$ and $\lambda_3<0$ therein; in addition the quadratic polynomial $t \mapsto A(t,1)$ may admit real roots only in case $(a_1)$ or $(a_2)$,
according to the sign of the function $\lambda_2-\lambda_3$. This observation may be slightly generalised in order 
to determine the directions on which the Jacobi operator is pure trace.
\begin{Pro} \label{J-alg}
Assume that $w \in T_mM$ satisfies $g(w,w)=1$ and $\ring{\J}(w)=0$. Then, up to sign, 
\begin{equation*}
w=\sqrt{\frac{\lambda_3}{\lambda_2}}e_1 \pm \sqrt{\frac{\lambda_2-\lambda_3}{\lambda_2}}e_2 \ when \ \mathrm{\lambda_2<\lambda_3} \ \mathrm{or} \ w=\sqrt{\frac{\lambda_2}{\lambda_3}}e_1 \pm \sqrt{\frac{\lambda_3-\lambda_2}{\lambda_3}}e_3 \ when \ \mathrm{\lambda_3<\lambda_2}.
\end{equation*} 
When $\lambda_2=\lambda_3$ we have $w=\pm e_1$.
\end{Pro}
\begin{proof}
We need to solve the quadratic equation $R(x,v)v=\lambda x$ for all $x \in w^{\perp}$. Since $\ric(w,w)=2\lambda$ we may write 
$\Ric w=2\lambda w+w_1$ with $w_1 \in w^{\perp}$. Using \eqref{R-p} together with $\rho=\Ric-\frac{\sca}{4}\Id$ then shows that $\ring{\J}(w)=0$ is equivalent to 
\begin{equation*}
\Ric(x)=g(w_1,x)w+(\frac{\sca}{2}-\lambda)x
\end{equation*}
for all $x \in w^{\perp}$. Observe that $\frac{\sca}{2}-\lambda \neq 0$, otherwise $\ric(x,x)=0$ on $w^{\perp}$; since the non-zero 
eigenvalues of $\Ric$ are both negative on $U_2$ this would lead to $w^{\perp} \subseteq \rm span \{e_1\}$, a contradiction. Since $\det(\Ric)=0$ there exists $tw+x_1 \in \ker \Ric$, where $t \in \mathbb{R}$ and $x_1 \in w^{\perp}$. It follows that 
$$ 2t\lambda+g(w_1,x_1)=0 \ \mathrm{and} \ tw_1+(\frac{\sca}{2}-\lambda)x_1=0.
$$
Because $\frac{\sca}{2}-\lambda \neq 0$ we must have $t\neq 0$ hence 
\begin{equation} \label{nw}
2\lambda(\frac{\sca}{2}-\lambda)=g(w_1,w_1)
\end{equation}
and $(\frac{\sca}{2}-\lambda)w -w_1 \in \ker \Ric$. If $w_2 \in w^{\perp}$ is orthogonal to $w_1$ and unit length 
then $\Ric w_2=(\frac{\sca}{2}-\lambda)w_2$; by trace considerations the other non-zero eigenvalue of $\Ric$ is 
$\frac{\sca}{2}+\lambda$. Since $\Ric(w_1)=(\frac{\sca}{2}-\lambda)(2\lambda w+w_1)$ we find that 
$$ \Ric(2\lambda w+w_1)=(\frac{\sca}{2}+\lambda)(2\lambda w+w_1).
$$
There are two cases to consider now.\\
$\bullet$ When $\frac{\sca}{2}+\lambda=\lambda_2$ and thus $\frac{\sca}{2}-\lambda=\lambda_3$. Then 
$$ (\frac{\sca}{2}-\lambda)w -w_1=xe_1, \ 2\lambda w+w_1=ye_2.
$$
Since $w$ has unit length, by also using \eqref{nw} we find $x^2=(\frac{\sca}{2}-\lambda)(\frac{\sca}{2}+\lambda)=\lambda_2\lambda_3$ 
as well as $y^2=2\lambda(\frac{\sca}{2}+\lambda)=(\lambda_2-\lambda_3)\lambda_2$. Since $w$ and $w_1$ are given by $(\frac{\sca}{2}+\lambda)w=xe_1+ye_2$ and 
$(\frac{\sca}{2}+\lambda)w_1=-2\lambda xe_1+(\frac{\sca}{2}-\lambda)y e_2$ the claim follows.\\
$\bullet$ When $\frac{\sca}{2}+\lambda=\lambda_3$ and thus $\frac{\sca}{2}-\lambda=\lambda_2$. This is entirely similar and left to the reader.
\end{proof}
Thus there are, up to sign, $2$ directions in $TM$ on which the Jacobi operator is diagonal. We now take advantage of this fact to obtain 
information on the algebraic structure of $\nabla \ric$.
\begin{Lem} \label{co-3}On the open subset of $U_2$ where $\lambda_2<\lambda_3$ we have 
\begin{itemize}
\item[(i)] $\mathscr{L}_{e_2}\lambda_2=\frac{2\lambda_2\lambda_3}{\lambda_2-\lambda_3}g(\nabla_{e_1}e_1,e_2)$
\item[(ii)] $\mathscr{L}_{e_1}\lambda_2=2\lambda_2g(\nabla_{e_2}e_1,e_2)$
\item[(iii)] for directions $X=xe_1+ye_2$ with $x,y \in \mathbb{R}$ the polynomial $D_1(X)=y^3\mathrm{d}^{12}_1(\frac{x}{y})$ where 
the polynomials $\mathrm{d}_1^{12}, \mathbf{a}^{12}$ in $\mathbb{R}[t]$ read  
\begin{equation} \label{d12-u}
\mathrm{d}_1^{12}=\frac{4\lambda_2}{\lambda_2-\lambda_3}g(\nabla_{e_1}e_1,e_2)\mathbf{a}^{12} \ \mathrm{and} \ 2\mathbf{a}^{12}=(\lambda_3-\lambda_2)t^2+\lambda_3.
\end{equation} 
\end{itemize}

\end{Lem}
\begin{proof}
Consider the vector fields $X=Ae_1$ and $Y=Be_2$ where the coefficients $A=\sqrt{\frac{\lambda_3}{\lambda_2}}$ and 
$B=\sqrt{\frac{\lambda_2-\lambda_3}{\lambda_2}}$. Since $\ring{\J}(X \pm Y)=0$ by Proposition \ref{J-alg}, equation \eqref{1-jet1} forces 
\begin{equation*}
(\nabla_{X+Y}\ric)(X+Y,X+Y)=(\nabla_{X-Y}\ric)(X-Y,X-Y)=0.
\end{equation*}
After expansion taking into account $(\nabla_{U}\ric)(e_1,e_1)=0$ and that $AB \neq 0$ over the region where $\lambda_2 \neq \lambda_3$ this yields 
\begin{equation} \label{ric-13}
\begin{split}
&B^2 (\nabla_{e_2}\ric)(e_2,e_2)+2A^2(\nabla_{e_1}\ric)(e_1,e_2)=0\\
& (\nabla_{e_1}\ric)(e_2,e_2)+2(\nabla_{e_2}\ric)(e_1,e_2)=0.
\end{split}
\end{equation}
The claims in (i) and (ii) follow now by expansion taking into the eigenvalue structure of the Ricci tensor. \\
(iii) Expansion of $D_1(X)$, taking into account that $(\nabla_{U}\ric)(e_1,e_1)=0$ whenever $U \in TM$ shows that 
\begin{equation*}
\begin{split}
D_1(X)=&2x^2y (\nabla_{e_1}\ric)(e_1,e_2)+xy^2 \left ( 2(\nabla_{e_2}\ric)(e_1,e_2)+(\nabla_{e_1}\ric)(e_2,e_2)\right )\\
&+ y^3(\nabla_{e_2}\ric)(e_2,e_2).
\end{split}
\end{equation*}
Using \eqref{ric-13} we find 
$D_1(X)=2y(x^2-\frac{A^2}{B^2}y^2) (\nabla_{e_1}\ric)(e_1,e_2)=-\frac{2\lambda_2}{\lambda_3-\lambda_2}g(\nabla_{e_1}e_1,e_2)y\mathbf{a}_{12}$, as claimed.
\end{proof}

For further use we record that the differential Bianchi identity, evaluated on  the eigenframe $ \{e_1,e_2,e_3\}$ reads
\begin{equation} \label{B2-bis}
\begin{split}
&\frac{1}{2}\Li_{e_1}\sca=\lambda_2g(\nabla_{e_2}e_2,e_1)+\lambda_3g(\nabla_{e_3}e_3,e_1)\\
&\frac{1}{2}\Li_{e_2}(\lambda_3-\lambda_2)=(\lambda_3-\lambda_2)g(\nabla_{e_3}e_3,e_2)-\lambda_2g(\nabla_{e_1}e_1,e_2)\\
& \frac{1}{2}\Li_{e_3}(\lambda_3-\lambda_2)=(\lambda_3-\lambda_2) g(\nabla_{e_2}e_2,e_3)+\lambda_3g(\nabla_{e_1}e_1,e_3).
\end{split}
\end{equation}
We now start describing the second order obstructions, starting with the following 
\begin{Lem}\label{Ric111}
We have 
\begin{equation*}
\ric(\nabla_{e_1}e_1,\nabla_{e_1}e_1)=-\frac{(\la_2-\la_3)^2}{2} \ \mathrm{in} \ U_2.
\end{equation*}
\end{Lem} 
\begin{proof}
Since $\Ric(e_1)=0$ we know that $u(e_1)=0$ and hence the second order derivative $(\nabla^2_{e_1,e_1}\rm ric)(e_1,e_1)=-2\tr (J(e_1) \circ J(e_1))$ according to 
\eqref{jet-2}. As in the proof of Proposition \ref{U1} this leads to $\rm ric(\nabla_{e_1}e_1,\nabla_{e_1}e_1)=
-\tr (J(e_1) \circ J(e_1))$. The claim follows now from Proposition \ref{J2}, by a short direct computation.
\end{proof}

The last ingredient we need before being able to fully investigate the polynomial constraint in Theorem \ref{obs2}
is the explicit matrix of the derived Jacobi operator $\ring{\J}^{\prime}$.
\begin{Pro} \label{jet-12}
Let $v$ be a locally defined unit vector field in $U_2$ and let $w_1,w_2$ be a local orthonormal frame in $v^{\perp}$ w.r.t. 
which $\ring{\J}(v)$ is diagonal, with corresponding matrix $\left ( \begin{array}{lr} A & 0 \\ 0 & -A \end{array} \right )$. The matrix of 
$\ring{\J}^{\prime}$ in the frame $\{w_1,w_2\}$ is then $\left ( \begin{array}{lr} A_1 & B_1 \\ B_1 & -A_1 \end{array} \right )$
where 
\begin{equation*}
\begin{split}
&A_1=\Li_vA-\ric (v, \nabla_vv)+2g(\nabla_vv,w_1)\ric( v,w_1)\\
&B_1=2Ag(\nabla_vw_1,w_2)+g(\nabla_vv,w_1)\ric (v,w_2)+g(\nabla_vv,w_2)\ric( v,w_1).
\end{split}
\end{equation*} 
\end{Pro}
\begin{proof}
Differentiating the eigenvalue equation $R(w_1,v)v=(A+\frac{t}{2})w_1$ where $t=g(\Ric v,v)$ yields 
\begin{equation*}
\begin{split}
(\nabla_vR)(w_1,v)v&=\Li_v(A+\frac{t}{2})w_1+(A+\frac{t}{2})\nabla_vw_1-R(\nabla_vw_1,v)v\\
&-R(w_1,\nabla_vv)v-R(w_1,v)\nabla_vv.
\end{split}
\end{equation*}
Now take into account that 
 $\nabla_vv \in \mathrm{span}\{w_1,w_2\}$ and $\nabla_vw_1 \in \mathrm{span}\{v,w_2\}$ . We get 
 $$ R(\nabla_vw_1,v)v=g(\nabla_vw_1,w_2)R(w_2,v)v=g(\nabla_vw_1,w_2)(-A+\frac{t}{2})w_2.
 $$
 Similarly, after expanding $\nabla_vv$ in the frame $\{w_1,w_2\}$ and re-arranging terms 
 \begin{equation*}
 \begin{split}
 R(w_1,\nabla_vv)v+R(w_1,v)\nabla_vv=&g(\nabla_vv,w_2)\left ( R(w_1,w_2)v+R(w_1,v)w_2 \right )\\
 &+g(\nabla_vv,w_1)R(w_1,v)w_1.
 \end{split}
 \end{equation*}
 Using that $\J(v)$ is diagonal w.r.t. $\{w_1,w_2\}$ and some easy trace arguments leads to 
 \begin{equation*}
\begin{split}
&R(w_1,v)w_1=-(A+\frac{t}{2})v-g(\Ric v,w_2)w_2\\
&R(w_1,w_2)v=g(\Ric v,w_2)w_1-g(\Ric v,w_1)w_2, \ R(w_1,v)w_2=\ric(v,w_2)w_1.
\end{split} 
 \end{equation*}
Gathering terms thus yields 
$$ \J^{\prime}(v)w_1=\left ( \Li_v(A+\frac{t}{2})-2g(\nabla_vv,w_2)\ric( v,w_2)\right )w_1+B_1w_2.
$$
Since $(\nabla_v\ric)(v,v)=\Li_vt-2\ric(\nabla_vv,v)$ and $\tr \ring{\J}^{\prime}(v)=(\nabla_v\ric)(v,v)$ it follows that 
$\ring{\J}^{\prime}(v)w_1=A_1w_1+B_1w_2$, where the coefficient functions $A_1,B_1$ read as stated. The claim follows by taking into 
account that $\ring{\J}^{\prime}(v): v^{\perp} \to v^{\perp} $ is symmetric and trace-free.
\end{proof}
\subsection{Expliciting the polynomial constraints} \label{Pyth}
Throughout this section we are dealing with homogeneous polynomials in $2$ variables; recall that any homogeneous 
polynomial $P\in \mathbb{R}[x,y]$ with $\deg P=m$ is determined from $P=y^mp(\frac{x}{y})$ for some $p \in \mathbb{R}[t]$ with $\deg p=m$.

We consider special directions $X$ in $TM$ according to the instances $(a_1)-(a_3)$ in the previous section and investigate the structure of the quantities featuring in \eqref{C3}.

The polynomial $A=y^2\mathbf{a}(\frac{x}{y})$ is not identically zero over $U_2$ and reads 
$$ 2\mathbf{a}=(\lambda_3-\lambda_2)t^2+\lambda_3 \ \mathrm{in \ case} \ (a_1) \ \ \mathrm{respectively} \ \ 2\mathbf{a}=(\lambda_2-\lambda_3)t^2+\lambda_2 \ \mathrm{in \ case} \ (a_2)
$$
according to Lemma \ref{jac-f}.
To analyse the properties of $B_1$, which turns out to have linear coefficients in terms the connection coefficients of the frame $\{e_1,e_2,e_3\}$, consider the triple $\mathbf{p}_{12}, \mathbf{p}_{13}, \mathbf{p}_{23}$ in $\mathbb{R}[t]$ given by 
\begin{equation*}
\begin{split}
\mathbf{p}_{12}:=&(\lambda_2-\lambda_3)g(\nabla_{e_1}e_2,e_3)t^2+
g((\lambda_2-\lambda_3)\nabla_{e_2}e_2+\lambda_3\nabla_{e_1}e_1,e_3)t+\lambda_3g(\nabla_{e_2}e_1,e_3)\\
\mathbf{p}_{13}:=&(\lambda_2-\lambda_3)g(\nabla_{e_1}e_2,e_3)t^2+g((\lambda_3-\lambda_2)\nabla_{e_3}e_3+\lambda_2\nabla_{e_1}e_1,e_2)t+
\lambda_2g(\nabla_{e_3}e_1,e_2)\\
\mathbf{p}_{23}:=&\lambda_3g(\nabla_{e_2}e_3,e_1)t^2+g(\lambda_3\nabla_{e_3}e_3-\lambda_2\nabla_{e_2}e_2,e_1)t-\lambda_2g(\nabla_{e_3}e_2,e_1).
\end{split}
\end{equation*}
Those polynomials fully determine $B_1$ as showed below.
\begin{Lem} \label{gcd1}
We have
\begin{itemize}
\item[(i)]  $B_1=\frac{1}{\sqrt{x^2+y^2}}y^4\mathbf{b}_1(\frac{x}{y})$ and $A_1=y^3\mathbf{a}_1(\frac{x}{y})$ where $\mathbf{b}_1, \mathbf{a}_1 \in \mathbb{R}[t]$ have degree at most $4$ respectively $3$
\item[(ii)] $\mathbf{b}_1=-(t^2+1)\mathbf{c}$ where 
\begin{equation} \label{c}
\mathbf{c}=\mathbf{p}_{12} \ \mathrm{in \ case} \ (a_1), \ \mathbf{c}=\mathbf{p}_{13} \ \mathrm{in \ case} \ (a_2), \ \mathbf{c}=\mathbf{p}_{23} \ \mathrm{in \ case} \ (a_3).
\end{equation}
 \end{itemize}
\end{Lem}
\begin{proof}
We only prove these statements for $X=xe_1+ye_2$ as the other two cases are entirely similar. Letting 
$v:=(x^2+y^2)^{-\frac{1}{2}}X$ we have $\ring{\J}^{\prime}(X)=(x^2+y^2)^{\frac{3}{2}}\ring{\J}^{\prime}(v)$. By Lemma 
\ref{jac-f} the matrix 
of $\ring{\J}{v}=(x^2+y^2)^{-1}\ring{\J}(X)$ is $(x^2+y^2)^{-1}\left ( \begin{array}{lr} A & 0 \\ 0 & -A \end{array} \right )$ in the orthonormal 
frame $w_1=e_3, w_2=(x^2+y^2)^{-\frac{1}{2}}(ye_1-xe_2)$. Now we apply Proposition \ref{jet-12} for the frame $\{v,w_1,w_2\}$; 
since $\mathrm{ric}(v,w_1)=0$ it follows that the matrix of $\ring{\J}^{\prime}(X)$ w.r.t. $\{w_1,w_2\}$ is $\left ( \begin{array}{lr} A_1 & B_1 \\ B_1 & -A_1 \end{array} \right )$ where 
\begin{equation} \label{A1B1}
\begin{split}
A_1=&(x^2+y^2)^{\frac{3}{2}} \left (\frac{1}{x^2+y^2}\Li_vA-\ric(v,\nabla_vv) \right )=\Li_XA-\rm ric(X,\nabla_XX)
\end{split}
\end{equation}
the off-diagonal matrix coefficient
\begin{equation*}
\begin{split}
B_1=&(x^2+y^2)^{\frac{3}{2}} \left (\frac{2A}{x^2+y^2}g(\nabla_vw_1,w_2)+g(\nabla_vv,w_1)\mathrm{ric}(v,w_2) \right )=\frac{1}{\sqrt{x^2+y^2}}y^4\mathbf{b}_1(\frac{x}{y})
\end{split}
\end{equation*}
and $\mathbf{b}_1(t)=2\mathbf{a}g(\nabla_{te_1+e_2}e_3,e_1-te_2)+g(\nabla_{te_1+e_2}(te_1+e_2),e_3)\ric(te_1+e_2,e_1-te_2).$ Clearly 
$A_1$ is homogeneous of degree $3$, hence $A_1=y^3\mathbf{a}_1(\frac{x}{y})$ as stated. Since the curvature term 
$\ric(te_1+e_2,e_1-te_2)=-t\lambda_2$ the factorisation  $\mathbf{b}_1=-(t^2+1)\mathbf{p}_{12}$ follows by a purely algebraic calculation.
\end{proof}
As these will be needed later on we also bring the polynomials $\mathbf{a}_1^{13}$ and $\mathbf{a}_1^{23}$ into ready to use form.
\begin{Lem} \label{a13s}
We have 
\begin{equation*}
\begin{split}
\mathbf{a}_1^{13}=&\biggl (3\mathbf{p}_{23}^{\prime}(0)-4\lambda_3 g(e_1,\nabla_{e_3}e_3) \biggr )t^3+ \biggl (\mathbf{p}_{12}^{\prime}(0)-
3\lambda_3g(\nabla_{e_1}e_1,e_3) \biggr )t^2+\mathbf{p}_{23}^{\prime}(0)t\\
+& \mathbf{p}_{12}^{\prime}(0)-2\lambda_3 g(\nabla_{e_1}e_1,e_3)+\frac{1}{2}\rm d_1^{13}(0).
\end{split}
\end{equation*}
\end{Lem}
\begin{proof}
Since $2\mathbf{a}^{13}=(\lambda_2-\lambda_3)t^2+\lambda_2$ and $\Ric e_1=0$ we have 
\begin{equation*}
\begin{split}
\mathbf{a}^{13}_1=&\frac{1}{2}\Li_{te_1+e_3}\left ((\lambda_2-\lambda_3)t^2+\lambda_2 \right )-\lambda_3 g(e_3, \nabla_{te_1+e_3}(te_1+e_3))\\
=&\frac{1}{2}\Li_{e_1}(\lambda_2-\lambda_3)t^3+\left (\frac{1}{2}\Li_{e_3}(\lambda_2-\lambda_3)-\lambda_3g(e_3,\nabla_{e_1}e_1) \right )t^2\\
+&\left (\frac{1}{2}\Li_{e_1}\lambda_2+\lambda_3g(e_1,\nabla_{e_3}e_3) \right)t+\frac{1}{2}\Li_{e_3}\lambda_2.
\end{split}
\end{equation*}
To express the coefficient of $t^3$ we use part (ii) in Lemma \ref{co-3} and \eqref{B2-bis} to check that 
\begin{equation*}
\begin{split}
\frac{1}{2}\Li_{e_1}(\lambda_2-\lambda_3)=&-\frac{1}{2}\Li_{e_1}\sca+\Li_{e_1}\lambda_2=-3\lambda_2g(\nabla_{e_2}e_2,e_1)-\lambda_3g(\nabla_{e_3}e_3,e_1)\\
&=3\mathbf{p}_{23}^{\prime}(0)-4\lambda_3 g(e_1,\nabla_{e_3}e_3).
\end{split}
\end{equation*}
 The coefficient of $t^2$ turns out to read as stated by using only \eqref{B2-bis}. For the coefficient of $t$ we have $\frac{1}{2}\Li_{e_1}\lambda_2+\lambda_3g(e_1,\nabla_{e_3}e_3)=-\lambda_2g(e_1,\nabla_{e_2}e_2)
+\lambda_3g(e_1,\nabla_{e_3}e_3)=\mathbf{p}_{23}^{\prime}(0)$, by part (ii) in Lemma \ref{co-3}. Finally 
$$\frac{1}{2}\Li_{e_3}\lambda_2=\frac{1}{2}\Li_{e_3}(\lambda_2-\lambda_3)+\frac{1}{2}\Li_{e_3}\lambda_3=\mathbf{p}_{12}^{\prime}(0)-2\lambda_3 g(\nabla_{e_1}e_1,e_3)+\frac{1}{2}\rm d_1^{13}(0)$$ by using again \eqref{B2-bis}.
\end{proof}
Similarly, we also record that 
\begin{Lem} \label{a1s}
We have 
\begin{equation*}
\begin{split}
-\mathbf{a}_1^{23}=&
\left (\frac{\lambda_2(3\lambda_3-2\lambda_2)}{\lambda_2-\lambda_3}g(\nabla_{e_1}e_1,e_2)+\mathbf{p}_{13}^{\prime}(0) \right )t^3
+\left (\frac{1}{2}\rm d_1^{13}-\frac{1}{4}\rm (d_1^{13})^{\prime \prime}-
\mathbf{p}_{12}^{\prime} \right )(0)t^2\\
+&\left (\frac{\lambda_2^2}{\lambda_2-\lambda_3}g(\nabla_{e_1}e_1,e_2)-\mathbf{p}_{13}^{\prime}(0) \right )t+\left (\frac{1}{2}\rm d_1^{13}+\frac{1}{2}\rm (d_1^{13})^{\prime \prime}+
\mathbf{p}_{12}^{\prime} \right )(0).
\end{split}
\end{equation*}
\end{Lem}
\begin{proof}
According to equation \eqref{A1B1} in the proof of Proposition \ref{jet-12} and since we also know that 
$-2\mathbf{a}^{23}=\lambda_3t^2+\lambda_2$ we get
$-\mathbf{a}_1^{23}=\frac{1}{2}\Li_{te_2+e_3}(\lambda_3t^2+\lambda_2)+{\rm ric}(te_2+e_3, \nabla_{te_2+e_3}(te_2+e_3))$. 
By algebraic expansion we obtain 
\begin{equation*}
\begin{split}
-\mathbf{a}_1^{23}=&(\frac{1}{2}\Li_{e_2}\lambda_3)t^3+\left (\frac{1}{2}\Li_{e_3}\lambda_3+
(\lambda_3-\lambda_2)g(\nabla_{e_2}e_2,e_3) \right )t^2\\
+&\left (\frac{1}{2}\Li_{e_2}\lambda_2+
(\lambda_2-\lambda_3)g(\nabla_{e_3}e_3,e_2) \right )t+\frac{1}{2}\Li_{e_3}\lambda_2.
\end{split}
\end{equation*}
As in the proof of Lemma \ref{a13s} this is brought into the form claimed in the statement by using Lemma \ref{co-3}, (i) and the Bianchi identities in \eqref{B2-bis}.
\end{proof}
Moving on to the rest of the quantities involved in \eqref{C3} we first consider directions $X$ of the type $(a_1)$ or $(a_2)$. 
Since $(\nabla_{e_1}\ric)(e_1,e_1)=0$ and 
$(\nabla^{2}_{e_1,e_1}\ric)(e_1,e_1)=0$ both $\dd_2=(\nabla^2_{X,X}\ric)(X,X)$ and $\dd_1=(\nabla_X\ric)(X,X)$ are multiples of $y$ hence we may write $\dd_2=y^4\mathrm{d}_2(\frac{x}{y})$ and $\dd_1=y^3\mathrm{d}_1(\frac{x}{y})$, with $\mathrm{d}_2$ and $\mathrm{d}_1$ in $\mathbb{R}[t]$ of degree $3$ respectively $2$.

\begin{Pro} \label{pol-1}
The constraint in \eqref{C3} for the choices of directions as in $(a_1)$ respectively $(a_2)$ reads 
\begin{equation} \label{P-eqn}
\mathrm{ P}^2+(t^2+1)(\mathrm{ d}_1\mathbf{c})^2=\Lambda(t^2+1)(\mathbf{a}\mathbf{c})^2
\end{equation} 
for real valued polynomials of degrees $\deg \mathrm{P} \leq 5, \deg \rm d_1 \leq 2, \deg \mathbf{c} \leq 2, \deg \mathbf{a}=2$ and where  $\rm P=\mathbf{a}\rm d_2-\mathbf{a}_1 \rm d_1$. In addition $\Lambda \in \mathbb{R}$ satisfies 
$\Lambda=-8\lambda_2$ in case $(a_1)$ and $\Lambda=-8\lambda_3$ in case $(a_2)$. 
\end{Pro}
The proof simply consists in plugging into \eqref{C3} the values for the homogeneous polynomials $\mathrm{D},\dd_1,A_1,B_1$ etc. as obtained above. 
There remains to investigate the content of \eqref{C3} for directions $X=xe_2+ye_3$ which correspond to the case $(a_3)$. Then 
$y$ does not necessarily factor out from $\dd_1$ and $\dd_2$ hence 
$\dd_2=y^4{\rm d}_2(\frac{x}{y})$ and $\dd_1=y^3{\rm d}_1(\frac{x}{y})$, 
with $\rm d_2$ and $\rm d_1$ in $\mathbb{R}[t]$ of degree at most $4$ respectively $3$. Since also $\ric(X,X)=\lambda_2x^2+\lambda_3y^2$ we obtain 
\begin{Pro} \label{pol-1bis}
The constraint in \eqref{C3} for choices of directions as in $(a_3)$  reads 
\begin{equation} \label{P-eqn23}
\mathrm{P}^2+(t^2+1)(\mathrm{ d}_1\mathbf{c})^2=-8(t^2+1)(\mathbf{a}\mathbf{c})^2(\lambda_2t^2+\lambda_3)
\end{equation} 
for polynomials in $\mathbb{R}[t]$ 
of degrees $\deg \rm P \leq 6, \deg \rm d_1 \leq 3, \deg \mathbf{c} \leq 2, \deg \mathbf{a}=2$ and where  $\rm P=\mathbf{a}\rm d_2-\mathbf{a}_1 \rm d_1$.
\end{Pro}
\subsection{Solving the polynomial constraints} \label{pol-s}
This section is devoted to solving the quadratic polynomial equations encountered previously. We begin with 
directions of type $(a_1)$ or $(a_2)$, in which case the main observation is that in most instances the polynomial $\rm d_1$ can be 
explicitly determined. 
\begin{Pro} \label{solve-1}
The solutions to \eqref{P-eqn} are 
\begin{itemize}
\item[(i)] $\mathbf{c}=0$, when $\rm P=0$
\item[(ii)] $\mathbf{c} \neq 0$ and $\mathrm{ d}_1=\pm \sqrt{\Lambda}\mathbf{a}$ when again $\rm P=0$
\item[(iii)] $\mathbf{c} \neq 0$ and $\lambda_3 < \lambda_2$ when 
$$ \pm \mathrm{ d}_1=\sqrt{-2\lambda_2}\left ( (\lambda_2-\lambda_3)t^2+2\lambda_2-\lambda_3 \right )  \ \mathrm{in\ case \ (a_1)}
$$
\item[(iv)] $\mathbf{c} \neq 0$ and $\lambda_2 < \lambda_3$ when 
$$ \pm \mathrm{ d}_1=\sqrt{-2\lambda_3}\left ((\lambda_3-\lambda_2)t^2+2\lambda_3-\lambda_2 \right ) \ \mathrm{in\ case \ (a_2)}.$$
\end{itemize}
\end{Pro}
\begin{proof}
If $\mathbf{c}=0$ then clearly $\rm P=0$ so we assume that $\mathbf{c} \neq 0$ from now on.
If $\mathbf{c}$ admits a real root $\lambda$ then $\mathrm{P}(\lambda)=0$; if $\mathbf{c}$ admits a complex 
root $z$ then $\mathrm{P}(z)=0$ hence $\mathrm{P}(\overline{z})=0$ since $\mathrm{P}$ has real coefficients. Including the case when $\deg \mathbf{c}=0$, when 
$\mathbf{c}$ is a non-zero constant, it follows that $\mathrm{P}=\mathbf{c}q$ where 
$q \in \mathbb{R}[t]$ satisfies $q^2+(t^2+1){\rm d}_1^2=\Lambda(t^2+1)\mathbf{a}^2$. 
Then $q(\pm i)=0$ so 
$q=(t^2+1)v$ where 
$(t^2+1)v^2=\Lambda \mathbf{a}^2-\rm d_1^2=\left ( \sqrt{\Lambda}\mathbf{a}-\rm d_1\right )
\left ( \sqrt{\Lambda}\mathbf{a}+\rm d_1 \right )$. Since $\deg \mathbf{a}=2$ and $\deg \rm d_1 \leq 2$ it follows that 
$\deg v \leq 1$. 

If $v=0$ we have $\rm d_1=\pm \sqrt{\Lambda}\mathbf{a}$ which corresponds to (ii). If 
$\deg v=1$ that is $v=x(t-\lambda)$ where $x \neq 0$ we have $\mathrm{gcd}(t-\lambda, t^2+1)=1$ so 
either $$ \sqrt{\Lambda}\mathbf{a}-\mathrm{ d}_1=x_1(t-\lambda)^2, \ \sqrt{\Lambda}\mathbf{a}+\mathrm{ d}_1=x_2(t^2+1)$$ 
or 
\begin{equation} \label{prel-s}
\sqrt{\Lambda}\mathbf{a}-\mathrm {d}_1=x_1(t^2+1), \ \sqrt{\Lambda}\mathbf{a}+\mathrm{ d}_1=x_2(t-\lambda)^2
\end{equation}
where $x_1,x_2 \in \mathbb{R}$ satisfy $x_1x_2=x^2>0$. In the first case $2\sqrt{\Lambda}\mathbf{a}=x_1(t-\lambda)^2+x_2(t^2+1)$;
as granted by Lemma \ref{jac-f}, the polynomial $\mathbf{a}$ reads $2\mathbf{a}=(\lambda_3-\lambda_2)t^2+\lambda_3$ in case $(a_1)$, respectively 
$2\mathbf{a}=(\lambda_2-\lambda_3)t^2+\lambda_2$ in case $(a_2)$, which allows equating coefficients. It follows that $\lambda=0$ in both cases, and that 
$x_1=-\lambda_2\sqrt{\Lambda}>0, x_2=\lambda_3\sqrt{\Lambda}<0$ in case $(a_1)$ respectively $x_1=-\lambda_3\sqrt{\Lambda}>0, x_2=\lambda_2\sqrt{\Lambda}<0$ in case $(a_2)$. As these solutions satisfy $x_1x_2<0$ the case 
when $\deg v=1$ cannot occur. Similarly, the system in\eqref{prel-s} has no solutions with $x_1x_2>0$.

The case when $\deg v=0$, when $v$ is constant but non-zero, corresponds to 
$$ \sqrt{\Lambda}\mathbf{a}-\mathrm{ d}_1=x_1(t^2+1) , \ \sqrt{\Lambda}\mathbf{a}+\mathrm {d}_1=x_2$$ 
or 
$$\sqrt{\Lambda}\mathbf{a}-\mathrm{ d}_1=x_1 , \ \sqrt{\Lambda}\mathbf{a}+\mathrm{ d}_1=x_2(t^2+1)
$$ 
where $x_1,x_2 \in \mathbb{R}$ satisfy $x_1x_2>0$. This is solved as above with solutions corresponding to (iii) and 
(iv) in the statement. 
\end{proof}
Next we move on to solving \eqref{P-eqn23}, which is of an entirely different nature and more difficult to handle.
\begin{Pro} \label{solve-2}
Assume that $\lambda_3 >\lambda_2$. The solutions to \eqref{P-eqn23} are 
\begin{itemize}
\item[(i)] $\mathbf{c}=0$, when $\rm P=0$
\item[(ii)] $\mathbf{c} \neq 0$ and 
\begin{equation*}
\mathrm{ d}_1= \pm \sqrt{2(\lambda_3-\lambda_2)}t(\lambda_3t^2+\lambda_2) \ \mathrm{and} \ \mathrm{P}=\pm \sqrt{-2\lambda_3}\mathbf{c}(t^2+1)
(\lambda_3t^2+\lambda_2).
\end{equation*}
\end{itemize}
\end{Pro} 
\begin{proof}
If $\mathbf{c}=0$ then clearly $\mathrm{P}=0$. Assuming now that $\mathbf{c} \neq 0$ we factorise $\mathrm{P}=\mathbf{c}(t^2+1)q$, as in the proof 
of Proposition \ref{solve-1} to find 
\begin{equation*} \label{so-23}
(t^2+1)q^2+\rd_1^2=-8\mathbf{a}^2 (\lambda_2t^2+\lambda_3)=-2(\lambda_3t^2+\lambda_2)^2 (\lambda_2t^2+\lambda_3)
\end{equation*}
where $\deg q \leq 2$. Observe this constraint further factorises according to 
\begin{equation} \label{fact-2}
\begin{split}
&-(t^2+1)\left (q-\sqrt{-2\lambda_2}(\lambda_3t^2+\lambda_2) \right ) \left (q+\sqrt{-2\lambda_2}(\lambda_3t^2+\lambda_2) \right )\\
&=\rd_1^2+2(\lambda_3-\lambda_2)(\lambda_3t^2+\lambda_2)^2\\
&=H_1H_2
\end{split}
\end{equation} 
where the complex valued polynomials 
$$H_1:=\rd_1-i\sqrt{2(\lambda_3-\lambda_2)}(\lambda_3t^2+\lambda_2) \ \mathrm{and} \ H_2:=\rd_1+i\sqrt{2(\lambda_3-\lambda_2)}(\lambda_3t^2+\lambda_2).
$$ 
In particular, it follows that $i$ is a root of either $H_1$ or 
of $H_2$. Assuming the former, write 
$H_1=(t-i)H$ where $H  \in \mathbb{C}[t]$ has degree at most $2$. After conjugation, $H_2=(t+i)\overline{H}$ so the constraint 
in \eqref{fact-2} becomes 
\begin{equation} \label{prod}
-\left (q-\sqrt{-2\lambda_2}(\lambda_3t^2+\lambda_2) \right ) \left (q+\sqrt{-2\lambda_2}(\lambda_3t^2+\lambda_2) \right )=H \overline{H}.
\end{equation}
Also record for further use that 
\begin{equation} \label{prod-10}
2i\sqrt{2(\lambda_3-\lambda_2)}(\lambda_3t^2+\lambda_2)=H_2-H_1=(t+i)\overline{H}-(t-i)H.
\end{equation}
The real valued polynomials $q_{\pm}:=q \pm \sqrt{-2\lambda_2}(\lambda_3t^2+\lambda_2)$ have degree at most two; since the r.h.s.
of \eqref{fact-2} is positive none of them may admit real roots, in particular $\deg q_{\pm} \in \{0,2\}$. Assuming that 
$\deg q_{+}=\deg q_{-}=2$, it follows that $\deg H=2$ as well. The roots of $q_{+}$ respectively $q_{-}$ must be thus of the form $z_{+}, \overline{z}_{+}$ respectively $z_{-}, \overline{z}_{-}$ with $z_{\pm}$ in $\mathbb{C} \backslash \mathbb{R}$. It follows that 
$Hz_{+}=H\overline{z}_{+}=Hz_{-}=H\overline{z}_{-}=0$ and since $H$ has degree $2$ we must have $z_{-}=z_{+}$ or 
$z_{-}=\overline{z}_{+}$. Equivalently $q_{+}$ and $q_{-}$ must be proportional fact which easily leads to 
$q=c\sqrt{-2\lambda_2}(\lambda_3t^2+\lambda_2)$ and $H=C(\lambda_3t^2+\lambda_2)$ where $c \in \mathbb{R}$ and $C \in \mathbb{C}$. 
According to \eqref{prod-10}, it follows that 
\begin{equation*}
2i\sqrt{2(\lambda_3-\lambda_2)}(\lambda_3t^2+\lambda_2)=
\left( \overline{C}(t+i)-C(t-i) \right )(\lambda_3t^2+\lambda_2)
\end{equation*}
whence $C=\sqrt{2(\lambda_3-\lambda_2)}$. From $-q_{+}q_{-}=H\overline{H}$ we thus get $c^2=\frac{\lambda_3}{\lambda_2}$; this produces the first set of solutions in (ii).

There remains to treat the lower degree cases, starting with $\deg q_{+}=0$ and $\deg q_{-}=2$. Then 
$q+\sqrt{-2\lambda_2}(\lambda_3t^2+\lambda_2)=c_1 \in \mathbb{R}$, we have $\deg H =1$ so we may write 
$H=\alpha t+\beta$ with $\alpha, \beta \in \mathbb{C}$. Identifying coefficients in \eqref{prod-10} shows that 
 \begin{equation*}
 \begin{split}
 \overline{\alpha}-\alpha=2i\lambda_3 \sqrt{2(\lambda_3-\lambda_2)}, \ \beta-i\alpha \in \mathbb{R}, 
 \beta+\overline{\beta}=2\lambda_2\sqrt{2(\lambda_3-\lambda_2)}.
 \end{split}
 \end{equation*}

Identifying coefficients in \eqref{prod} also shows that 
 \begin{equation*}
 \begin{split}
 &\vert \alpha \vert^2=c_1\lambda_32\sqrt{-2\lambda_2}, \ \beta \overline{\alpha}+\alpha \overline{\beta}=0\\
 &\vert \beta \vert^2=-c_1 \left (c_1-2\lambda_2 \sqrt{-2\lambda_2} \right ).
 \end{split}
 \end{equation*}  
From the first five equations above we find that 
$\alpha=-i\lambda_3 \sqrt{2(\lambda_3-\lambda_2)}, c_1=\frac{\lambda_3(\lambda_3-\lambda_2)}{\sqrt{-2\lambda_2}}$ and 
$\beta=\lambda_2 \sqrt{2(\lambda_3-\lambda_2)}$. Plugging these in the equation giving $\vert \beta \vert^2$ above leads to the additional constraint $(r-1)(r^2+4)=0$, where $r=\frac{\lambda_3}{\lambda_2}$.  As this equation has no solution $0 < r <1$ we obtain a contradiction. 

To show that the case $\deg q_{+}=2$ and $\deg q_{-}=0$ cannot occur one proceeds in an entirely similar way. Since $\deg(q_{+}-q_{-})=2$ we cannot have $\deg q_{+}=\deg q_{-}=0$ and the claim is proved under the assumption that $H_1(i)=0$. When $H_2(i)=0$ we have 
$$ -q_{+}q_{-}=H\overline{H}, \ -2i\sqrt{2(\lambda_3-\lambda_2)}(\lambda_3t^2+\lambda_2)=H_1-H_2=(t+i)\overline{H}-(t-i)H
$$
for some $H \in \mathbb{C}[t]$ with $\deg H \leq 2$. The same reasoning as above shows that the only solution is 
$H=-\sqrt{2(\lambda_3-\lambda_2)}(\lambda_3t^2+\lambda_2)$ and the claim is fully proved.
\end{proof}
When $\lambda_3 <\lambda_2$ solving \eqref{P-eqn23} directly is much more complicated; the main reason is that the left hand side of 
\eqref{fact-2} factorises only over $\mathbb{R}$ and thus has a more involved root structure. Motivated by Remark \ref{rmk1} we however observe that after the transformation  
$$ (\mathrm{ P},{\rm d}_1,\mathbf{c},\mathbf{a}) \mapsto (\tilde{\rm P}:=t^6 {\rm P}(\tfrac{1}{t}),\tilde{\rm d}_1:=t^3 {\rm d}_1(\tfrac{1}{t}), 
\tilde{\mathbf{c}}:=t^2\mathbf{c}(\tfrac{1}{t}), \tilde{\mathbf{a}}:=t^2\mathbf{a}(\tfrac{1}{t}))
$$
the constraint \eqref{P-eqn23} becomes $\tilde{\mathrm{P}}^2+(t^2+1)(\tilde{\rd}_1 \tilde{\mathbf{c}})^2=-8(t^2+1)(\tilde{\mathbf{a}}\tilde{\mathbf{c}})^2(\lambda_3t^2+\lambda_2)
$. Since $-2\widetilde{\mathbf{a}}=\lambda_2t^2+\lambda_3$ this corresponds to 
swapping $\lambda_2$ and $\lambda_3$ and hence we may apply Proposition \ref{solve-2} to solve the case $\lambda_3<\lambda_2$ as well.
\section{Differential consequences} \label{diff}
The assumptions in this section are again the same as in the statement of Theorem \ref{main-intro0}.
The approach now consists in investigating how the solutions to the polynomial equations obtained previously may overlap on open subsets 
of $U_2$. Throughout this section we work over the open region $U^{+}_2$ in $U_2$ where $\lambda_2 < \lambda_3$. For further 
use we first record the following
\begin{Lem} \label{dis}The following hold
\begin{itemize}
\item[(i)] the polynomial $\rm d_1^{23}$ satisfies 
\begin{equation} \label{d123-bis}
\begin{split}
\rd_1^{23}=& \frac{2\lambda_2\lambda_3}{\lambda_2-\lambda_3}g(\nabla_{e_1}e_1,e_2)t^3+\left (\Li_{e_3}\lambda_2-2(\lambda_3-\lambda_2)g(\nabla_{e_2}e_2,e_3) \right )t^2\\
+&\left (\Li_{e_2}\lambda_3-2(\lambda_2-\lambda_3)g(\nabla_{e_3}e_3,e_2) \right )t+\Li_{e_3}\lambda_3
\end{split}
\end{equation}
\item[(ii)] the polynomial $\rm d_1^{13}$ satisfies 
\begin{equation} \label{d114}
\rd_1^{13}=-2\lambda_3g(\nabla_{e_1}e_1,e_3)t^2+(\Li_{e_1}\lambda_3+2\lambda_3g(\nabla_{e_3}e_3,e_1))t+\Li_{e_3}\lambda_3.
\end{equation}
\end{itemize}
\end{Lem}
\begin{proof}
(i) follows after taking into account that, after expansion, 
\begin{equation} \label{d123}
\begin{split}
&\rd_1^{23}=(\nabla_{te_2+e_3}\mathrm{ric})(te_2+e_3,te_2+e_3)\\
=& (\nabla_{e_2}\ric)(e_2,e_2)t^3+\left ((\nabla_{e_3}\ric)(e_2,e_2)+2(\nabla_{e_2}\ric)(e_2,e_3) \right )t^2\\
&+
\left ((\nabla_{e_2}\ric)(e_3,e_3)+2(\nabla_{e_3}\ric)(e_2,e_3) \right )t+(\nabla_{e_3}\ric)(e_3,e_3).
\end{split}
\end{equation}
Since $\{e_1,e_2,e_3\}$ is an eigenframe for the Ricci tensor with corresponding eigenfunctions $0,\lambda_2,\lambda_3$ by further 
expanding $\nabla\ric$ we see that all coefficients read as stated and that $(\nabla_{e_2}\ric)(e_2,e_2)=\Li_{e_2}\lambda_2$. The claim 
follows by using Lemma \ref{co-3},(i).\\
(ii) Similarly, expanding $\rd_1^{13}=(\nabla_{te_1+e_3}\mathrm{ric})(te_1+e_3,te_1+e_3)$ leads to 
\begin{equation*} \label{d113}
\begin{split}
\rd_1^{13}=&2(\nabla_{e_1}\ric)(e_1,e_3)t^2+\left ((\nabla_{e_1}\ric)(e_3,e_3)+2(\nabla_{e_3}\ric)(e_1,e_3) \right )t+(\nabla_{e_3}\ric)(e_3,e_3)
\end{split}
\end{equation*}
and the claim follows again by using the eigenvalue structure of $\Ric$.
\end{proof}
These explicit expressions will be used in what follows to compare coefficients in $\mathrm{d}_1^{13}$ and $\mathrm{d}_1^{12}$ as in most cases these polynomials assume explicit algebraic forms granted by Proposition \ref{solve-1} and Proposition \ref{solve-2}. 
\begin{Pro} \label{diff-1}
We have $\mathbf{p}_{23}=0$ over $U_2^{+}$.
\end{Pro}
\begin{proof}
By combining the identities in \eqref{d12-u} and \eqref{d123-bis} we observe that the leading respectively the free coefficient in 
$\mathrm{d}_1^{23}$ respectively $\mathrm{d}_1^{12}$ are equal, that is $\frac{1}{6}(\mathrm{d}_1^{23})^{(3)}(0)=\mathrm{d}_1^{12}(0)$. 

Arguing by contradiction, assume that the open subset where $\mathbf{p}_{23} \neq 0$ is not empty. On this region the polynomial $\rm d_1^{23}$ is given by 
$\rd_1^{23}=\pm \sqrt{2(\lambda_3-\lambda_2)}t(\lambda_3t^2+\lambda_2)$, according to Proposition \ref{solve-2}; in particular 
the leading coefficient in $\mathrm{d}_1^{23}$ is $ \pm \sqrt{2(\lambda_3-\lambda_2)}\lambda_3$, which thus 
must equal the free coefficient in $\rm d_1^{12}$. As the solutions to \eqref{P-eqn} in Proposition \ref{solve-1}, (ii) have free coefficient 
$\pm \sqrt{-2\lambda_2}\lambda_3 \neq \pm \sqrt{2(\lambda_3-\lambda_2)}\lambda_3$ and those in (iii) of the same proposition require 
$\lambda_2>\lambda_3$ we may conclude that $\mathbf{p}_{12}=0$. To fully outline the information encoded in the expression for $\mathrm{d}_1^{23}$ we use \eqref{d123-bis} and also Lemma \ref{co-3},(i) to obtain 
\begin{equation} \label{sys-23}
\begin{split}
&\Li_{e_2}\lambda_2=\varepsilon \lambda_3 \sqrt{2(\lambda_3-\lambda_2)}\\
&\Li_{e_3}\lambda_2-2(\lambda_3-\lambda_2)g(\nabla_{e_2}e_2,e_3)=0\\
&\Li_{e_2}\lambda_3+2(\lambda_3-\lambda_2)g(\nabla_{e_3}e_3,e_2)=\varepsilon \lambda_2 \sqrt{2(\lambda_3-\lambda_2)}\\
&\Li_{e_3}\lambda_3=0
\end{split}
\end{equation} 
where $\varepsilon \in \{-1,1\}$. The last and second equation above ensure that the Lie derivative $\frac{1}{2}\Li_{e_3}(\lambda_3-\lambda_2)=
-(\lambda_3-\lambda_2)g(\nabla_{e_2}e_2,e_3)$; after comparison with the differential Bianchi identity \eqref{B2-bis} this yields 
$2(\lambda_3-\lambda_2)g(\nabla_{e_2}e_2,e_3)+\lambda_3g(\nabla_{e_1}e_1,e_3)=0$. Because the polynomial $\mathbf{p}_{12}$ vanishes 
identically we also have $-(\lambda_3-\lambda_2)g(\nabla_{e_2}e_2,e_3)+\lambda_3g(\nabla_{e_1}e_1,e_3)=0$, hence 
\begin{equation*}
g(\nabla_{e_2}e_2,e_3)=g(\nabla_{e_1}e_1,e_3)=0.
\end{equation*}
Similarly, combining the first and and third equations in the system \eqref{sys-23} leads to having $\frac{1}{2}\Li_{e_2}(\lambda_3-\lambda_2)=
-(\lambda_3-\lambda_2)g(\nabla_{e_3}e_3,e_2)+\frac{\varepsilon}{2}(\lambda_2-\lambda_3)\sqrt{2(\lambda_3-\lambda_2)}$. After comparison with the differential Bianchi identity \eqref{B2-bis} we find 
\begin{equation} \label{33-2}
2(\lambda_3-\lambda_2)g(\nabla_{e_3}e_3,e_2)-\lambda_2g(\nabla_{e_1}e_1,e_2)=\frac{\varepsilon}{2}(\lambda_2-\lambda_3)\sqrt{2(\lambda_3-\lambda_2)}.
\end{equation}
Since $g(\nabla_{e_1}e_1,e_3)=0$ the identity in \eqref{d114} ensures that $\deg \rm d_1^{13} \leq 1$. As the solutions to \eqref{P-eqn} listed in Proposition \ref{solve-1}, (ii) and (iv)  have degree $2$ we conclude that $\mathbf{p}_{13}=0$. 
The vanishing of $\mathbf{p}_{13}$ guarantees that, in particular,   
$(\lambda_3-\lambda_2)g(\nabla_{e_3}e_3,e_2)
+\lambda_2g(\nabla_{e_1}e_1,e_2)=0$. Thus 
\begin{equation*}
g(\nabla_{e_1}e_1,e_2)=-\frac{\varepsilon}{6\lambda_2}(\lambda_2-\lambda_3)\sqrt{2(\lambda_3-\lambda_2)}
\end{equation*}
by also using \eqref{33-2}.
Since we have just seen that $g(\nabla_{e_1}e_1,e_3)=0$, it follows that $\ric(\nabla_{e_1}e_1,\nabla_{e_1}e_1)=\frac{(\lambda_3-\lambda_2)^3}{18\lambda_2}$. Lemma \ref{Ric111} then forces 
$\la_3 = -8 \la_2$, a contradiction as $\la_2<\la_3<0$, see also Lemma \ref{vec}. This finishes showing that the polynomial $\mathbf{p}_{23}$ vanishes identically.
\end{proof}

\begin{Pro} \label{diff-2}
On the open region in $U_2$ where $\lambda_2 < \lambda_3$ we must have $\mathbf{p}_{12}\mathbf{p}_{13}=0$.
\end{Pro}
\begin{proof}
Assume that the open subset where $\mathbf{p}_{12}\mathbf{p}_{13} \neq 0$ is not empty. Since $\lambda_2<\lambda_3$, Proposition \ref{solve-1} ensures that $\varepsilon_1 \rd_1^{12}=\sqrt{-2\lambda_2}\left ( (\lambda_3-\lambda_2)t^2+\lambda_3 \right )$ where 
$\varepsilon_1 \in \{ -1,1 \}$. Within the region under scrutiny we have $\mathbf{p}_{13} \neq 0$ so according to Proposition \ref{solve-1} the only possibilities for $\rd_1^{13}$ are 
$\varepsilon_2 \rd_1^{13}= \sqrt{-2\lambda_3} \left ( (\lambda_2-\lambda_3)t^2+\lambda_2 \right )$ with $\varepsilon_2 \in \{-1,1\}$ or $\varepsilon_3 \rd_1^{13}=\sqrt{-2\lambda_3}\left ((\lambda_3-\lambda_2)t^2+2\lambda_3-\lambda_2 \right )
$ where $\varepsilon_3 \in \{-1,1\}$ as well. Accordingly, there are two cases to distinguish. \\
{\bf{Case I:}} $\varepsilon_1 \mathrm{d}_1^{12}=\sqrt{-2\lambda_2}\left ( (\lambda_3-\lambda_2)t^2+\lambda_3 \right )$ and 
$\varepsilon_2 \mathrm{d}_1^{13}= \sqrt{-2\lambda_3} \left ( (\lambda_2-\lambda_3)t^2+\lambda_2 \right )$. \\
These requirements hold on the open set where $\rd_1^{13} \neq \pm \sqrt{-2\lambda_3}\left ((\lambda_3-\lambda_2)t^2+2\lambda_3-\lambda_2 \right )$.
The coefficient of $t^2$ in $\rm d_1^{12}$ respectively $\rm d_1^{13}$
is $-2\lambda_2 g(\nabla_{e_1}e_1,e_2)$ respectively $-2\lambda_3 g(\nabla_{e_1}e_1,e_3)$ according to \eqref{d12-u} respectively \eqref{d114}. It follows that 
\begin{equation*}
g(\nabla_{e_1}e_1,e_2)=\varepsilon_1 \frac{\lambda_3-\lambda_2}{\sqrt{-2\lambda_2}} \ \mathrm{and}  \ g(\nabla_{e_1}e_1,e_3)=\varepsilon_2 \frac{\lambda_2-\lambda_3}{\sqrt{-2\lambda_3}}.
\end{equation*}
Therefore $\ric(\nabla_{e_1}e_1,\nabla_{e_1}e_1)=-(\lambda_2-\lambda_3)^2$. Using Lemma \ref{Ric111} we see that $\la_2=\la_3$, a contradiction. This case may not occur.
$\\$
{\bf{Case II:}} $\varepsilon_1 \mathrm{ d}_1^{12}=\sqrt{-2\lambda_2}\left ( (\lambda_3-\lambda_2)t^2+\lambda_3 \right )$ and 
$\varepsilon_3 \mathrm{ d}_1^{13}=\sqrt{-2\lambda_3}\left ((\lambda_3-\lambda_2)t^2+2\lambda_3-\lambda_2 \right )
$.
An argument identical to that of Case I shows that this case may also not occur.

Summarising, we have obtained a contradiction in each of the cases above, which shows that $\mathbf{p}_{12}\mathbf{p}_{13}=0$ as claimed.
\end{proof}
\subsection{The second order derivative of the Ricci tensor} \label{2nd}

The following simple observation will be used repeatedly in this section.
\begin{Lem} \label{mid}
We have $(\rm d_1^{13})^{\prime}(0)=0$ if and only if
$$\Li_{e_1}\lambda_2=\Li_{e_1}\lambda_3=g(\nabla_{e_3}e_3,e_1)=g(\nabla_{e_2}e_2,e_1)=0.$$
\end{Lem}
\begin{proof}
The requirement $\Li_{e_1}\lambda_3+2\lambda_3g(\nabla_{e_3}e_3,e_1)=0$, when combined with Lemma \ref{co-3},(ii) yields 
$\frac{1}{2}\Li_{e_1}\sca+\lambda_3g(\nabla_{e_3}e_3,e_1)+\lambda_2g(\nabla_{e_2}e_2,e_1)=0$. After comparison with the Bianchi 
identity 
\eqref{B2-bis} it follows that $\frac{1}{2}\Li_{e_1}\sca=\lambda_3g(\nabla_{e_3}e_3,e_1)+\lambda_2g(\nabla_{e_2}e_2,e_1)=0$. As we also know that $\mathbf{p}_{23}=0$ the claim follows.
\end{proof}

The main idea to generate additional obstructions on the connection coefficients is to take further advantage of the fact that 
$\mathbf{p}_{23}=0$. Those involve a priori second order derivatives of the Ricci tensor since having $\rm P^{23}=0$ reads 
$$ \rm {\mathbf{a}}^{23}d_2^{23}-\mathbf{a}_1^{23}\rm d_1^{23}=0
$$
by Proposition \ref{pol-1}. This contains however first order algebraic information; since the roots of the quadratic polynomial 
$\rm {\mathbf{a}}^{23}$ are $\pm i\sqrt{\frac{\lambda_2}{\lambda_3}}$ it follows that 
\begin{equation*}
\rm d_1^{23}\left ({\it i}\sqrt{\frac{\lambda_2}{\lambda_3}} \right )=0 \ \mathrm{or} \ \mathbf{a}_1^{23}\left ({\it i}\sqrt{\frac{\lambda_2}{\lambda_3}} \right )=0.
\end{equation*}
In the rest of this section we will explore the geometric content of each of the those root equations, more precisely their impact on the structure of the connection coefficients of the Ricci eigenframe $\{e_1,e_2,e_3\}$.  
The following general identities essentially relate the coefficients of 
$\rm d_1^{23}$ to those of $\rm d_1^{13}$ respectively $\mathbf{a}_1^{23}$ in a ready to use way.
\begin{eqnarray}\label{Re}
\begin{split}
\rm Re \ \rm d_1^{23}\left ({\it i}\sqrt{\frac{\lambda_2}{\lambda_3}} \right )=&\frac{\lambda_3-\lambda_2}{\lambda_3}
\left (\rm d_1^{13}(0)+\frac{3\lambda_2}{\lambda_2-\lambda_3}\frac{(\rm d_1^{13})^{\prime \prime}(0)}{2} \right )
-\frac{4\lambda_2}{\lambda_3}\mathbf{p}_{12}^{\prime}(0) \\
\rm Im \ \rm d_1^{23}\left ({\it i}\sqrt{\frac{\lambda_2}{\lambda_3}}\right )=&4\left (\mathbf{p}_{13}^{\prime}(0)-2\lambda_2g(\nabla_{e_1}e_1,e_2)\right )\sqrt{\frac{\lambda_2}{\lambda_3}}. 
\end{split}
\end{eqnarray}
The proof is entirely analogous to that of the Lemma \ref{a13s} and Lemma \ref{a1s}; it consists of bringing to final form the expression for 
$\rm d_1^{23}$ in \eqref{d123-bis} by taking into account Lemma \ref{co-3},(ii) together with the Bianchi identity \eqref{B2-bis}.
\begin{Lem}\label{2d1} On any open subset of $U_2^{+}$ where $\rm d_1^{23}\left ({\it i}\sqrt{\frac{\lambda_2}{\lambda_3}} \right )=0$ we 
have $\mathbf{p}_{13}=0$.
\end{Lem}
\begin{proof}
Arguing by contradiction we assume that the open set where $\mathbf{p}_{13} \neq 0$ is not empty; according to Proposition \ref{diff-2} on this region we must have $\mathbf{p}_{12}=0$. We continue by analysing the possible algebraic occurences for $\rd_1^{13}$; recall those are listed in Proposition \ref{solve-1}, (ii) and (iv) and correspond to $\rd_1^{13}= \pm \sqrt{-2\lambda_3} \left ( (\lambda_2-\lambda_3)t^2+\lambda_2 \right )$ or $\rd_1^{13}=\pm \sqrt{-2\lambda_3}\left ((\lambda_3-\lambda_2)t^2+2\lambda_3-\lambda_2 \right )$.
Since in both cases $(\rd_1^{13})^{\prime}(0)=0$, the first equation in \eqref{Re} shows that requiring $i\sqrt{\frac{\lambda_2}{\lambda_3}}$ be a root of $\rd_1^{23}$ implies 
$\rd_1^{13}\left (\sqrt{\frac{3\lambda_2}{\lambda_2-\lambda_3}}\right )=0$. This corresponds to $\lambda_2=0$ respectively $\lambda_3=2\lambda_2$ for the first respectively the second instances in $\rm d_1^{13}$. None of these may occur over $U_2^{+}$, where $\lambda_2<\lambda_3<0$. We have thus obtained a contradiction hence the set where $\mathbf{p}_{13}\neq 0$ is empty and the claim is proved.
\end{proof}
These preparations allow to reduce to the study of specific roots of the cubic polynomial 
$\mathbf{a}_1^{23}$. To proceed in that direction record the general identity 
\begin{equation} \label{Re2}
2\rm Re \ \mathbf{a}_1^{13}\left ({\it i}\sqrt{\frac{\lambda_2}{\lambda_2-\lambda_3}} \right)=\frac{2\lambda_3}{\lambda_3-\lambda_2}\mathbf{p}_{12}^{\prime}(0)-\frac{\lambda_2+2\lambda_3}{\lambda_2-\lambda_3}\frac{(\rm d_1^{13})^{\prime \prime}(0)}{2}+\rm d_1^{13}(0)
\end{equation}
which follows from Lemma \ref{a13s}.
\begin{Pro}{\red{\label{2d1p}}} We have $\mathbf{a}_1^{23}\left (i\sqrt{\frac{\lambda_2}{\lambda_3}}\right )=0$ on $U_2^{+}$.
\end{Pro}
\begin{proof}
This is again by contradiction; we assume that the open subset of $U_2^{+}$ where $\mathbf{a}_1^{23}\left (i\sqrt{\frac{\lambda_2}{\lambda_3}}\right )\neq 0$ is not empty. On this set we thus have $\rm d_1^{23}\left ({\it i}\sqrt{\frac{\lambda_2}{\lambda_3}} \right )=0$ and further 
$\mathbf{p}_{13}=0$ by Lemma \ref{2d1}. Over open sets where $\mathbf{p}_{13}=0$ the second equation 
in \eqref{Re} forces 
$$ g(\nabla_{e_1}e_1,e_2)=g(\nabla_{e_3}e_3,e_2)=0.
$$
In particular $\rm d_1^{12}=0$ by (iii) in Lemma \ref{co-3}. Proposition \ref{solve-1} then implies that $\mathbf{p}_{12}=0$ as well.
Further on, since $\mathbf{p}_{13}=0$ we know that $\rm P^{13}=0$ hence 
$$ \rm {\mathbf{a}}^{13}d_2^{13}-\mathbf{a}_1^{13}\rm d_1^{13}=0.
$$
Since the roots of $\rm {\mathbf{a}}^{13}$ are $\pm i\sqrt{\frac{\lambda_2}{\lambda_2-\lambda_3}}$, there are two cases to consider.\\
{\bf{Case I:}} When $\rm d_1^{13}\left (\it i\sqrt{\frac{\lambda_2}{\lambda_2-\lambda_3}} \right)=0$. Since $\deg \rm d_1^{13} \leq 2$ this forces the vanishing of $(\rm d_1^{13})^{\prime}(0)$ hence $\rm d_1^{13}\left (\sqrt{\frac{3\lambda_2}{\lambda_2-\lambda_3}}\right )=0$ by the first equation in \eqref{Re}. As a polynomial of degree at most two admitting a real and a purely imaginary root must vanish, we conclude that $\rm \rm d_1^{13}=0$; in particular $g(\nabla_{e_1}e_1, e_3)=0$ by the explicit expression for $\rm d_1^{13}$ in \eqref{d114}. It follows that $\nabla_{e_1}e_1=0$ hence Lemma 
\ref{Ric111} and 
having $\la_2<\la_3$ in $U_2^+$ 
show that this case may not occur.\\
{\bf{Case II:}} When $\mathbf{a}_1^{13}\left (i\sqrt{\frac{\lambda_2}{\lambda_2-\lambda_3}} \right)=0$. 
Since $\mathbf{p}_{12}=0$, after taking into account \eqref{Re2} together with the first equation in \eqref{Re} we end up with 
$$\rm d_1^{13}(0)=\frac{\lambda_2+2\lambda_3}{\lambda_2-\lambda_3}\frac{(\rm d_1^{13})^{\prime \prime}(0)}{2}=-\frac{3\lambda_2}{\lambda_2-\lambda_3}\frac{(\rm d_1^{13})^{\prime \prime}(0)}{2}.$$
As $2\lambda_2+\lambda_3<0$ in $U_2^{+}$ it follows that $(\rm d_1^{13})^{\prime \prime}(0)=0$ and hence 
$g(\nabla_{e_1}e_1,e_3)=0$. We obtain a contradiction exactly as in Case I above, which proves the claim.
\end{proof}
To determine in which circumstances $i\sqrt{\frac{\lambda_2}{\lambda_3}}$ may be a root of $\mathbf{a}_1^{23}$ we use the same approach as in Lemma \ref{2d1}, however based this time on 
the general identities 
\begin{eqnarray}\label{Re3}
\begin{split}
\rm Re \ \mathbf{a}_1^{23}\left ({\it i}\sqrt{\frac{\lambda_2}{\lambda_3}} \right )=&\frac{\lambda_2-\lambda_3}{2\lambda_3}\rm Re \ d_1^{13}
\left ({\it i}\sqrt{\frac{\lambda_2+2\lambda_3}{\lambda_2-\lambda_3}} \right )-\frac{\lambda_2+\lambda_3}{\lambda_3}\mathbf{p}_{12}^{\prime}(0)\\
-\sqrt{\frac{\lambda_3}{\lambda_2}}\rm Im \ \mathbf{a}_1^{23}\left ({\it i}\sqrt{\frac{\lambda_2}{\lambda_3}} \right )=&\frac{2\lambda_2^{2}}{\lambda_3}g(\nabla_{e_1}e_1,e_2)-\frac{\lambda_2
+\lambda_3}{\lambda_3}\mathbf{p}_{13}
^{\prime}(0)
\end{split}
\end{eqnarray}
which follow from Lemma \ref{a1s} by direct algebraic computation.
\begin{Lem}\label{2d2} We have $\mathbf{p}_{13}=0$ over $U_2^{+}$.
\end{Lem}
\begin{proof}
From Proposition \ref{2d1p} we know that $\mathbf{a}_1^{23}\left (i\sqrt{\frac{\lambda_2}{\lambda_3}}\right )=0$ on $U_2^{+}$. 
By contradiction, we work on the open set where $\mathbf{p}_{13}\neq 0$, which we assume to be non-empty. Arguments entirely similar to those in the proof of Lemma \ref{2d1} then yield $(\mathrm{d}_1^{13})^{\prime}(0)=0$. Furthermore, on the region under consideration 
$\mathbf{p}_{12}=0$ by Proposition \ref{diff-2}, hence the first equation in \eqref{Re3} shows that $\mathbf{a}_1^{23}\left (i\sqrt{\frac{\lambda_2}{\lambda_3}}\right )=0$ implies ${\rm d}_1^{13}\left (i\sqrt{\frac{\lambda_2+2\lambda_3}{\lambda_2-\lambda_3}}\right)=0$. However the possible instances for $\rm d_1^{13}$ as considered in the proof of Lemma \ref{2d1} satisfy 
$$ {\rm d}_1^{13}\left (i\sqrt{\frac{\lambda_2+2\lambda_3}{\lambda_2-\lambda_3}}\right)=\pm 2\sqrt{-2\lambda_3}\lambda_3 \ \mathrm{or} \ {\rm d}_1^{13}\left (i\sqrt{\frac{\lambda_2+2\lambda_3}{\lambda_2-\lambda_3}}\right)=\pm 4\sqrt{-2\lambda_3}\lambda_3.
$$
It follows that $\lambda_3=0$, which is a contradiction, hence $\mathbf{p}_{13}=0$ over $U_2^{+}$ as claimed.
\end{proof}
The main result of this section now reads
\begin{Thm} \label{main1}
Let $(M^3,g)$ be such that for any $(p,v) \in SM$ there exists a (possibly short time) solution $u=u(t) \in \rm Sym^2_0(\gamma^{\prime}(t)^{\perp})$ to the Riccati equation \eqref{main11i}, where $\gamma$ is the geodesic through $v$. Then the metric $g$ satisfies 
$\lambda_2=\lambda_3$ on $U_2$.
\end{Thm}
\begin{proof}
To summarise the information obtained so far, over the open set $U_2^{+}$ we have $\mathbf{a}_1^{23}\left (i\sqrt{\frac{\lambda_2}{\lambda_3}}\right )=0$ by Proposition \ref{2d1p} and also $\mathbf{p}_{13}=0$ by Lemma \ref{2d2}. We now assume that 
$U_2^{+}$ is not empty and work towards obtaining a contradiction, following the same line of argumentation as in the proof of Proposition \ref{2d1p}. Using this time the second equation in \eqref{Re3} together with $\mathbf{p}_{13}=0$ we find 
$$ g(\nabla_{e_1}e_1,e_2)=g(\nabla_{e_3}e_3,e_2)=0.
$$
Lemma \ref{co-3}, (iii) then yields $\mathrm{d}_1^{12}=0$ hence $\mathbf{p}_{12}=0$ by Proposition \ref{solve-1}. There remains to spell out the numerical content of the remaining eigenvalue equations for the polynomials $\rm d_1^{13}$ respectively $\mathbf{a}_1^{13}$.\\
{\bf{Case I:}} When ${\rm d}_1^{13}\left (i\sqrt{\frac{\lambda_2}{\lambda_2-\lambda_3}} \right)=0$. Since $\deg \rm d_1^{13} \leq 2$, we obtain $(\rm d_1^{13})^{\prime}(0)=0$, which implies ${\rm d}_1^{13}\left (i\sqrt{\frac{\lambda_2+2\lambda_3}{\lambda_2-\lambda_3}}\right )=0$ by the first equation in \eqref{Re3}. As the degree of $\rm d_1^{13}$ is at most two, we derive that ${\rm d}_1^{13}=0$. The rest of the arguments for showing that this case may not occur are identical to those in the proof of Case I in Proposition \ref{2d1p}.  \\
{\bf{Case II:}} When $\mathbf{a}_1^{13}\left (i\sqrt{\frac{\lambda_2}{\lambda_2-\lambda_3}} \right)=0$. First record the general identity 
\begin{equation*} \label{Im1}
    {\rm Im}\ \mathbf{a}_1^{13}\left (i\sqrt{\frac{\lambda_2}{\lambda_2-\lambda_3}} \right)=-\frac{1}{\lambda_2-\lambda_3}\sqrt{\frac{\lambda_2}{\lambda_2-\lambda_3}}\left((2\lambda_2+\lambda_3)\mathbf{p}_{23}^{\prime}(0)-4\lambda_2\lambda_3 g(\nabla_{e_3}e_3,e_1)\right)
\end{equation*}
which follows directly from Lemma \ref{a13s}. Since $\mathbf{p}_{23}=0$ by Proposition \ref{diff-1}, this forces $g(\nabla_{e_3}e_3,e_1)=g(\nabla_{e_2}e_2,e_1)=0$. The vanishing of the polynomials 
$\mathbf{p}_{12}$ etc. guarantees that $g(\nabla_{e_i}e_j,e_k)=0$ for $i \neq j \neq k$; taking this into account, a short computation also based on having $g(\nabla_{e_1}e_1,e_2)=g(\nabla_{e_2}e_2,e_1)=0$ shows that the curvature term $g(R(e_2,e_1)e_1,e_2)=-g(\nabla_{e_1}e_1,\nabla_{e_2}e_2)$. However 
$$g(\nabla_{e_1}e_1,\nabla_{e_2}e_2)=g(\nabla_{e_1}e_1,e_3)g(\nabla_{e_2}e_2,e_3)=
\frac{\lambda_3}{\lambda_3-\lambda_2}(g(\nabla_{e_1}e_1,e_3))^2
=-\frac{\lambda_3-\lambda_2}{2},
$$ after successively using that $\mathbf{p}_{12}=0$ and Lemma 
\ref{Ric111}. As we already know by \eqref{K2} that $g(R(e_2,e_1)e_1,e_2)=-\frac{\lambda_3-\lambda_2}{2}$, we obtain a contradiction, showing that the set $U_2^{+}$ is empty; an entirely analogous argument shows that the set $U_2^{-}$ is empty as well hence 
$\lambda_2=\lambda_3$ on $U_2$ as claimed.
\end{proof}

The section ends with the following clarification regarding \cite{KPS}.
\begin{Rem} \label{err2}
In \cite{KPS} the erroneous (see counterexample below) fact that,in dimension $3$, the eigenvectors of the Ricci tensor do come from an orthogonal coordinate system was assumed and used.
On $\mathbb{R}^3$ consider metrics $g$ with respect to which the co-frame $e=\{e^1,e^2,e^3\}$ is orthonormal and satisfies $\mathrm{d}e^1=\Lambda e^2 \wedge e^3$ for some $\Lambda \in \mathbb{R}, \Lambda \neq 0$ as well as $\mathrm{d}e^2=\mathrm{d}e^3=0$. This co-frame is an eigenframe for the Ricci tensor, with constant 
eigenfunctions $\{-\frac{\Lambda^2}{2},\frac{\Lambda^2}{2}, \frac{\Lambda^2}{2}\}$. However the co-frame $e$ does not induce an orthogonal coordinate system since $\mathrm{d}e^1 \wedge e^1$ does not vanish identically.
\end{Rem}
\section{Further geometry on $U_2$} \label{S5}
As we have just seen the Ricci tensor satisfies $\lambda_2=\lambda_3=\tfrac{1}{2}\sca$ on $U_2$. Equivalently, the curvature of $g$ has nullity, that is 
$R(e_1, \cdot)=0$. In dimension $3$ metrics with non-trivial Riemann nullity have been studied, under additional assumptions, in \cite{Bl}; see also \cite{MoSe} for results in the higher dimensional case. Without attempting to describe metrics 
with Riemann nullity in general, we investigate in this section how this curvature 
assumption interacts with the Riccati equation set-up by using the same algebraic approach as previousy. We begin with a few 
observations and notation set-up.

By Lemma \ref{Ric111} it follows that 
$\nabla_{e_1}e_1=0$. Further on we consider the distribution $\mathscr{H}:=e_1^{\perp}$ together with the shape operator 
$$ Q:TM \to TM \ \mathrm{where} \ QU:=\nabla_{U}e_1
$$
whenever $U \in TM$; since the vector field $e_1$ is totally geodesic we have $Q\mathscr{H} \subseteq \mathscr{H}$. This operator allows splitting the Levi-Civita connection according to 
$$ \nabla_XY=\widetilde{\nabla}_XY-g(QX,Y)e_1
$$
for all $X,Y$ in $\Gamma(\mathscr{H})$. By imposing $\widetilde{\nabla}e_1=0$ we obtain a metric connection 
in $TM$, which is actually the Bott connection of the foliation induced by the vector field $e_1$. We observe that the derivatives of the Ricci form may be computed according to the following 
\begin{Lem} \label{ders}The following hold whenever $X \in \Gamma(\mathscr{H})$
\begin{itemize}
\item[(i)] we have $(\nabla_X\ric)(X,X)=\tfrac{1}{2}g(X,X)(\mathrm{d}\sca)X $
\item[(ii)] the second order derivative of $\ric$ reads 
 \begin{equation}\label{2ndR-b}
 \begin{split}
 (\nabla_{X,X}^2\ric)(X,X)=&\tfrac{1}{2}g(X,X)(\widetilde{\nabla}_X\mathrm{d}\sca)X\\
 -&\sca g(QX,X) \left ( \tfrac{1}{2}\tr(Q)g(X,X)+
 g(QX,X)\right ). 
 \end{split}
 \end{equation}
\end{itemize}
\end{Lem}
\begin{proof}
(i) follows from the more general identity 
\begin{equation} \label{pR}
(\nabla_X\ric)(Y,Z)=\tfrac{1}{2}g(Y,Z)(\mathrm{d}\sca)X 
\end{equation}
for all 
$X,Y$ and $Z$ in $\Gamma(\mathscr{H})$. This is proved by differentiating in $\ric(Y,Z)=\tfrac{\sca}{2}g(Y,Z)$.\\
(ii) using (i) and also \eqref{pR} shows that 
\begin{equation*} 
\begin{split}
 (\nabla_{X,X}^2\ric)(X,X)=&X(\nabla_X\ric)(X,X)-(\nabla_{\nabla_XX}\ric)(X,X)-2(\nabla_X \ric)(\nabla_XX,X)\\
 =&\tfrac{1}{2}g(X,X)(\widetilde{\nabla}_X\mathrm{d}\sca)X+
 g(QX,X) \left ((\nabla_{e_1} \ric)(X,X)+2 (\nabla_{X} \ric)(e_1,X)\right )\\
 =&\tfrac{1}{2}g(X,X)(\widetilde{\nabla}_X\mathrm{d}\sca)X+g(QX,X) \left ( \tfrac{1}{2}\Li_{e_1}\sca g(X,X)-\sca g(QX,X) \right ).
 \end{split}
\end{equation*}
The claim follows by observing that the differential Bianchi identity \ref{B2-bis} reduces to 
\begin{equation} \label{BB3}
\Li_{e_1}\sca=-\sca \tr(Q).
\end{equation}
\end{proof}
Further on, we split $Q=S+A$ where $S \in \mathrm{Sym}^2\mathscr{H}$ and $A$ in $\Lambda^2\mathscr{H}$. All subsequent computations 
will be performed with respect to an orthonormal frame $\{e_2,e_3\}$ in $\mathscr{H}$ with respect to which $S$ is diagonal, that is 
$$ Se_2=\Lambda_2e_2  \ \mathrm{and} \ Se_3=\Lambda_3e_3.
$$
The algebraic arguments in Proposition \ref{solve-2} carry over the case $\lambda_2=\lambda_3$ and show that the possible solutions to \eqref{P-eqn23} are 
\begin{itemize}
\item[$\bullet$] $\mathbf{c}=0$, when $\mathrm{P}=0$
\item[$or$]
\item[$\bullet$] $\mathbf{c} \neq 0$, when $\mathrm{d}_1=0$ and $\mathrm{P}=\pm \tfrac{1}{2}\sca \sqrt{-\sca} \mathbf{c}(t^2+1)^2$.
\end{itemize}
\subsection{When $\mathbf{p}_{23} \neq 0$} \label{51}
In this section we work with directions $X=te_2+e_3 \in \mathscr{H}$ with respect to which we must have 
$\mathrm{d}_1^{23}=(\nabla_X\ric)(X,X)=0$. From Lemma \ref{ders}, (i) it follows that $\Li_{e_2}\sca=\Li_{e_3}\sca=0$. Since 
\begin{equation*}
\mathrm{P}^{23}=\mathbf{a}^{23}{\rm d}_2^{23}-\mathbf{a}_1^{23}{\rm d}_1^{23} \ \mathrm{ and} \ -4{\bf a}^{23}=\sca(t^2+1)
\end{equation*}
 the vanishing of $\mathrm{d}_1^{23}$ entails 
\begin{equation} \label{sol223}
\mathrm{d}_2^{23}=\pm 2 \sqrt{-\sca}\mathbf{p}_{23}(t^2+1).
\end{equation}
Also recall that according to the definitions we have 
$$\mathrm{d}_2^{23}= (\nabla_{X,X}^2\ric)(X,X)+2\tr (\ring{\J}(X) \circ (\ring{\J}(X)).$$
In addition, record that 
\begin{equation} \label{jac23}
\tr (\ring{\J}(X) \circ (\ring{\J}(X))=C\sca^2 (t^2+1)^2
\end{equation} 
where the constant 
$C=\tfrac{1}{8}>0$; this follows by direct computation from \eqref{J2}. At the same time part (ii) in Lemma 
\ref{ders} yields 
\begin{equation*} \label{2ndR}
\begin{split}
 (\nabla_{X,X}^2\ric)(X,X)
 =&g(QX,X) \left ( \tfrac{1}{2}\Li_{e_1}\sca g(X,X)-\sca g(QX,X) \right )
 \end{split}
\end{equation*}
for all $X \in \mathscr{H}$. Thus \eqref{sol223} 
is equivalent with the polynomial constraint 
\begin{equation}\label{2ndR}
\begin{split}
&(t^2\Lambda_2+\Lambda_3)(\tfrac{1}{2}\Li_{e_1}\sca(t^2+1)-\sca(t^2\Lambda_2+\Lambda_3))\\
=&(t^2+1)(-2C\sca^2 (t^2+1)\pm2\sqrt{-\sca}\mathbf{p}_{23}).
\end{split}
\end{equation}
This clearly forces 
$$\Lambda_2=\Lambda_3=\Lambda$$ thus 
$$ \mathbf{p}_{23}=-\frac{\sca}{2}g(Ae_2,e_3)(t^2+1)
$$
by taking into account the definition of $\mathbf{p}_{23}$. The remainder of \eqref{2ndR} therefore reads 
\begin{equation} \label{c123}
\Lambda(\tfrac{1}{2}\Li_{e_1}\sca-\sca\Lambda)=-2C\sca^2\pm \sca\sqrt{-\sca} g(Ae_2,e_3)
\end{equation}
which can be solved as outlined in the Lemma below. Record that we also have 
$$-2\sca \Lambda=\Li_{e_1}\sca$$
by the Bianchi identity as phrased in \eqref{BB3}.
\begin{Lem} \label{l51}
The following hold on the open subset of $U_2$ where $\mathbf{p}_{23} \neq 0$
\begin{itemize}
\item[(i)] we have $\Lambda=0$ and $\mathrm{d} \sca=0$
\item[(ii)]$e_1$ is a unit Killing vector field with $\mathscr{L}_{e_1}g=0$ and $\mathrm{d}e^1=\pm 4C \sqrt{-\sca}e^{23}$.
\end{itemize}
\end{Lem}
\begin{proof}
(i) since $\mathrm{d} \sca$ vanishes on $e_2$ and $e_3$ we have $g([e_2,e_3],e_1)\Li_{e_1} \sca=0$. However the specific algebraic form for $Q$ 
above yields 
$g([e_2,e_3],e_1)=2g(Ae_3,e_2)$. Since $\mathbf{p}_{23} \neq 0$ it follows that $A \neq 0$ 
which entails $\Li_{e_1} \sca=0$. Hence $\Lambda=0$ and further 
$\mathrm{d} \sca=0$ as claimed.\\
(ii) since $\Lambda=0$ equation \eqref{c123} grants $g(Ae_2,e_3)=\pm 2C \sqrt{-\sca}$ which leads to the claimed expression for 
$\mathrm{d}e^1$.
\end{proof}
This information suffices for drawing the final conclusion in this section.
\begin{Pro} \label{p52}
We have $\mathbf{p}_{23}=0$ on $U_2$.
\end{Pro}
\begin{proof}
We work on the open subset where $\mathbf{p}_{23} \neq 0$ which we assume not empty.Then $Q=\tfrac{1}{2}\mathrm{d}e^1$; at the same time 
, since $R(e_1, \cdot)=0$ we must have $(\nabla_{U_1}Q)U_2=(\nabla_{U_2}Q)U_1$ for all $U_1,U_2$ in $TM$. Because $Q$ is a closed 
$2$-form it follows that $\nabla Q=0$ or equivalently $\nabla (e^{23})=0$, after using part (ii) in Lemma \ref{51}. Applying the Hodge star thus yields $\nabla e_1=0$ hence $C \sqrt{-\sca}=0$ a contradiction.
\end{proof}
\subsection{Proof of theorem \ref{main-intro0}} \label{pf-main}
We begin with the observation that only two types of geometries may locally occur within $U_2$.
\begin{Lem} \label{l53}
The following hold on $U_2$
\begin{itemize}
\item[(i)] we have $Q=\Lambda \id_{\mathscr{H}}$ for some function $\Lambda$
\item[(ii)] the open subset of $U_2$ where $\Lambda \neq 0$ is locally a metric cone 
$(\mathbb{R}_{>0} \times \Sigma, dr^2+r^2g_{\Sigma})$ for some Riemann surface $(\Sigma,g_{\Sigma})$. In addition, the function 
$\Lambda$ satisfies $\Lambda^2=r^{-2}$
\item[(iii)] on any open subset of $U_2$ where $\Lambda=0$ we have that $\nabla e_1=0$; thus locally, such a subset 
is the Riemannian product $(\mathbb{R} \times \Sigma, dr^2+g_{\Sigma})$ for some Riemann surface $(\Sigma,g_{\Sigma})$.
\end{itemize}
\end{Lem}
\begin{proof}
Proposition \ref{p52} ensures that $\mathbf{p}_{23}=0$. The definition of $\mathbf{p}_{23}$ thus leads to 
$A=0$ and $\Lambda_2=\Lambda_3=\Lambda$; in other words $Q=\Lambda \mathrm{id}_{\mathscr{H}}$. From $R(e_1,\cdot)e_1=0$ we get 
$\nabla_{e_1}Q+Q^2=0$, equivalently 
\begin{equation*}
\Li_{e_1}\Lambda+\Lambda^2=0.
\end{equation*}
Also keep in mind that $-2\sca\,\Lambda=\Li_{e_1}\sca$, by the Bianchi identity. Work now on the open set set where 
$\Lambda \neq 0$ and consider the vector field $\mathbf{V}:=\tfrac{1}{\Lambda}e_1$. We have $\nabla_{X}\mathbf{V}=X$ when 
$X \perp e_1$; in addition, 
\begin{equation*}
\nabla_{e_1}\mathbf{V}=\Li_{e_1}(\tfrac{1}{\Lambda})e_1=e_1.
\end{equation*}
That is $\nabla^g\mathbf{V}=1_{TM}$. As it is well known, see e.g. \cite{AL}, it follows that $(U_2,g)$ is locally a metric cone 
$(\mathbb{R}_{>0} \times \Sigma, \mathrm{d}r^2+r^2g_{\Sigma})$ for some Riemann surface $(\Sigma,g_{\Sigma})$; moreover the cone vector field reads 
$\mathbf{V}=r\partial_r$. Record that, by construction 
$$ \Li_{\mathbf{V}}\sca=\tfrac{1}{\Lambda}\Li_{e_1}\sca=-2\,\sca.
$$
This is consistent with the general formula for the scalar curvature of metric cones which reads $\sca=r^{-2}(\sca_{\Sigma}-2)$. The expression for $\Lambda$ follows from $g(\mathbf{V}, \mathbf{V})=\Lambda^{-2}$.\\
Finally, on any open subset where $\Lambda=0$ we have $Q=0$ and the claim is fully proved.
\end{proof}
The rest of the proof consists in exploiting, at first in metric cone set-up in part (ii) of Lemma \ref{l53},
 the remaining algebraic constraint encoded in having $\rm P^{23}=0$. The latter reads 
\begin{equation} \label{prod-eqn}
\mathbf{a}^{23}{\rm d}_2^{23}=\mathbf{a}_1^{23}{\rm d}_1^{23}
\end{equation}
where $\mathrm{d}_2^{23}=2\tr (\ring{\J}(X) \circ \ring{\J}(X))+(\nabla^2_{X,X}\ric)(X,X)$.
We show below this is equivalent with a Hessian type equation on the scalar curvature $\sca_{\Sigma}$, which moreover can be explicitly solved. The list of the polynomial expressions involved in \eqref{prod-eqn} reads as follows
\begin{itemize}
\item[$\bullet$] $-4{\bf a}^{23}=\sca(t^2+1)$, see Lemma \ref{jac-f}
\item[$\bullet$] $2{\rm d}_1^{23}=\left ((\Li_{e_2}\sca) t+\Li_{e_3}\sca \right )(t^2+1)=(\mathrm{d} \sca)X (t^2+1)$; use the expansion in \eqref{d123}
\item[$\bullet$] $-4\mathbf{a}_1^{23}=(\mathrm{d} \sca)X (t^2+1)$,
see proof of Lemma \ref{a1s}
\item[$\bullet$] $\tr (\ring{\J}(X) \circ \ring{\J}(X))=C\sca^2(t^2+1)^2$ where the constant $C=\tfrac{1}{8}>0$, see \eqref{jac23}
\item[$\bullet$] $(\nabla^2_{X,X}\ric)(X,X)=\tfrac 12(\widetilde{\nabla}_X \mathrm{d}\sca)X (t^2+1)-2\sca \Lambda^2(t^2+1)^2$; use Lemma \ref{ders}.
\end{itemize}
The polynomial equation in \eqref{prod-eqn} then reads 
\begin{equation*} \label{hess}
\sca (\widetilde{\nabla}_X \mathrm{d}\sca)X-
(\mathrm{d}\sca \otimes \mathrm{d}\sca)(X,X)+4\sca^2(C\sca-\Lambda^2)g(X,X)=0.
\end{equation*}
Since $\Lambda^2=r^{-2}$ and $\sca=r^{-2}(\sca_{\Sigma}-2)<0$ is nonwhere vanishing 
the projection of this equation onto $\Sigma$ is 
\begin{equation} \label{hess-D}
D\mathrm{d} \ln (2-\sca_{\Sigma})=4(1+C(2-\sca_{\Sigma}))g_{\Sigma}
\end{equation}
where $D$ indicates the Levi-Civita connection of $g_{\Sigma}$. Here we have taken into account that 
$\widetilde{\nabla}$ is the horizontal lift of the Levi-Civita connection of $g_{\Sigma}$ to $\mathscr{H}$.\\

In an entirely similar way, on any open subset of $U_2$ where $\Lambda=0$ we have that, locally, the geometry is that of a Riemannian product $(\mathbb{R} \times \Sigma,\mathrm{d}r^2+g_{\Sigma})$ for some Riemann surface $(\Sigma,g_{\Sigma})$; see Lemma \ref{l53}, (iii). Then \eqref{prod-eqn} reads 
\begin{equation} \label{hess-Dp}
D\mathrm{d} \ln (2-\sca_{\Sigma})=4C(2-\sca_{\Sigma})g_{\Sigma}.
\end{equation}

All ingredients are now in place in order to give the \\
$\\$
{\bf{Proof of theorem \ref{main-intro0}.}}\\
We will show that the set $U_2$ is empty by showing that the Hessian equations in \eqref{hess-D} and \eqref{hess-Dp} admit 
no solution. We thus deal with a Riemann surface $(\Sigma, g_{\Sigma})$ admitting a locally defined function $\varphi:=
\ln(2-\sca_{\Sigma})$ satisfying 
\begin{equation*}
D(\mathrm{d}\varphi)=(C_0+\frac{1}{2}e^{\varphi})g_{\Sigma}
\end{equation*}
where $C_0 \in \{4,0\}$; here we taken into account that $C=\tfrac{1}{8}$. Differentiating shows that $D^2_{X,Y}\mathrm{grad} \varphi=\tfrac{1}{2}e^{\varphi}
g_{\Sigma}(\mathrm{grad}\varphi,X) Y$ hence 
$$2R^D(X,Y)\mathrm{grad}\varphi=e^{\varphi}\left (g_{\Sigma}(\mathrm{grad}\varphi,X) Y-
g_{\Sigma}(\mathrm{grad}\varphi,Y) X \right )$$ 
for all $X,Y \in T\Sigma$. Taking the trace and using that $\ric_{g_{\Sigma}}=\tfrac{\sca_{\Sigma}}{2}g_{\Sigma}$ shows that 
$$ -\sca_{\Sigma}\mathrm{d}\varphi=e^{\varphi}\mathrm{d}\varphi.
$$
Equivalently $\mathrm{d}\varphi=0$ which is prohibited since $C_0+\tfrac{1}{2}e^{\varphi}>0$. Therefore 
the equations \eqref{hess-D} and \eqref{hess-Dp} do not admit solutions, showing that the set $U_2$ is empty. Since 
$U_1$ is empty as well we conclude, by density, that $g$ is flat over $M$.

\subsection*{Conflict of interest}On behalf of all authors, the corresponding author \footnote{Paul-Andi Nagy} states that there is no conflict of interest.

\end{document}